\documentclass{article}

 \usepackage[preprint, nonatbib]{neurips_2020}

\usepackage[utf8]{inputenc} 
\usepackage[T1]{fontenc}    
\usepackage{hyperref}       
\usepackage{url}            
\usepackage{booktabs}       
\usepackage{amsfonts}       
\usepackage{nicefrac}       
\usepackage{microtype}      

\usepackage{amsmath}
\usepackage{amssymb}
\usepackage{multirow}
\usepackage{amsthm}
\usepackage{adjustbox}
\usepackage{thm-restate}
\usepackage{enumitem}
\usepackage{bbm}
\usepackage{algorithm,algorithmic}
\usepackage{amsmath,amssymb,bm,mathrsfs,mathtools}
\usepackage[numbers]{natbib}

\newtheorem{theorem}{Theorem}[section]

\newtheorem{assumptions}[theorem]{Assumption}
\newtheorem{proposition}[theorem]{Proposition}
\newtheorem{lemma}[theorem]{Lemma}
\newtheorem{corollary}[theorem]{Corollary}
\newtheorem{remark}[theorem]{Remark}
\theoremstyle{definition}

\newtheorem{example}[theorem]{Example}

\def\hC{{\widehat C}}

\def\bE {{\bf E}}

\def\mF{{\mathcal F}}

\def\bP{{\mathbbm{P}}}
\def\bE{{\mathbbm{E}}}

\newcommand{\real}{\mathbb{R}}

\allowdisplaybreaks

\title{Conformal prediction with localization}

\author{%
 Leying Guan\\
  Department of Biostatistics \\
  Yale University\\
  New Haven, CT 06510 \\
  \texttt{leying.guan@yale.edu} \\
}

\begin{document}

\maketitle

\begin{abstract}
  We propose a new method called  localized conformal prediction, where we can perform conformal inference using only a local region around a new test sample to construct its confidence interval.  Localized conformal inference is a natural extension to conformal inference. It generalizes the method of conformal prediction to the case where we can break the data exchangeability, so as to give the test sample a special role.  To our knowledge, this is the first work that introduces such a localization to  the framework of conformal prediction. We prove that our proposal can also have assumption-free and finite sample coverage guarantees, and we compare the behaviors of localized conformal prediction and conformal prediction  in simulations. 
\end{abstract}

\section{Introduction}

Let $Z_i\coloneqq (X_i, Y_i)\in \real^p\times \real $ for $i=1,\ldots, n$ be i.i.d regression data from some distribution $\mathcal{P}$. Let $Z_{n+1} = (X_{n+1}, Y_{n+1})$ be a new test sample with its response $Y_{n+1}$ unobserved. Given a nominal coverage level $\alpha$, we are interested in constructing confidence intervals (CI) $\hC(x)$, indexed by $x\in \real^p$, such that 
\begin{equation}
\label{eq:goal}
P(Y_{n+1}\in \hat C(X_{n+1})) \geq \alpha,\;\;\forall \mathcal{P}.
\end{equation}
The conformal inference is a framework for constructing $\hat C(x)$ satisfying eq.\ (\ref{eq:goal}), assuming only that $Z_{n+1}$ also comes from  $\mathcal{P}$ \cite{vovk2005algorithmic, shafer2008tutorial, vovk2009line,lei2014distribution, lei2018distribution}.

Conformal inference constructs CI based on a score function $V:\real^{p}\times \real \rightarrow [0, \infty)$. The score function measures how unlikely a sample is from distribution $\mathcal{P}$, and is constructed in a way such that $V_i = V(Z_i)$ are exchangeable with each other for $i = 1,\ldots, n+1$.  By exchangeability,  we know \cite{vovk2005algorithmic}
\begin{equation}
\label{eq:eq1}
\bP\left\{V_{n+1} \leq Q(\alpha; V_{1:n}\cup \{\infty\})\right\}\geq \alpha, \;\forall \mathcal{P}.
\end{equation}
where $Q(\alpha; V_{1:n}\cup \{\infty\})$ is the level $\alpha$ quantile of the empirical distribution of $\{V_1,\ldots, V_{n}, \infty\}$. Although the construction of $V$ can also be data-dependent, for illustration purposes, let's first consider a data-independent $V(.)$, and let $V(x, y) = |y - \mu(x)|$ where $\mu(x)$ is a fixed prediction function for the response $y\in \real$ at $x\in \real^p$.  To decide whether any value $y$ is  included in $\hC(X_{n+1})$, conformal inference tests the null hypothesis that $Y_{n+1} = y$  based on  eq.\ (\ref{eq:eq1}), and includes $y$ in $\hC(X_{n+1})$ if $V(z_{n+1}) \leq Q(\alpha; V_{1:n}\cup \infty)$, where $z_{n+1} = (X_{n+1}, y)$.

While it is good to have an almost assumption-free CI, the conformal CI treats all training samples equally regardless their distance to $X_{n+1}$. However, in some cases, we may want to emphasize more a local region around $X_{n+1}$. Such a localized approach is especially desirable when the distribution of $V(Z_{n+1})$ is heterogeneous across different values for $X_{n+1}$.   Consider the  example  $Y_i =X_i+\epsilon_i$ with $\epsilon_i |X_i \sim \frac{|X_i|}{|X_i|+1}N(0,1)$, and $X_i\sim Unif(-2,2)$ for $i = 1,2,\ldots, n+1$. We construct the CI for $Y_{n+1}$ by applying conformal inference to the score function $V(x, y) = |x-y|$. Figure \ref{fig:illustration1} shows the conformal confidence band using 1000 training samples (blue curves) and the underlying true confidence band (black curves) at level $\alpha = .95$.  The conformal confidence band can not capture the heterogeneity in a given score function $V(.)$ because it has treated all training samples equally for all test sample observations.
\begin{figure}
\vskip -.2in
\caption{\em Conformal bands (blue), localized conformal bands (red) and underlying true confidence bands (black) at level $\alpha = .95$. The conformal bands cannot capture the heterogeneity in the distribution of $V(X_{n+1}, Y_{n+1})$ for different $X_{n+1}$. The grey dots represent the actual test observations. }
\label{fig:illustration1}
\begin{center}
\includegraphics[width = .5\textwidth, height = .35\textwidth]{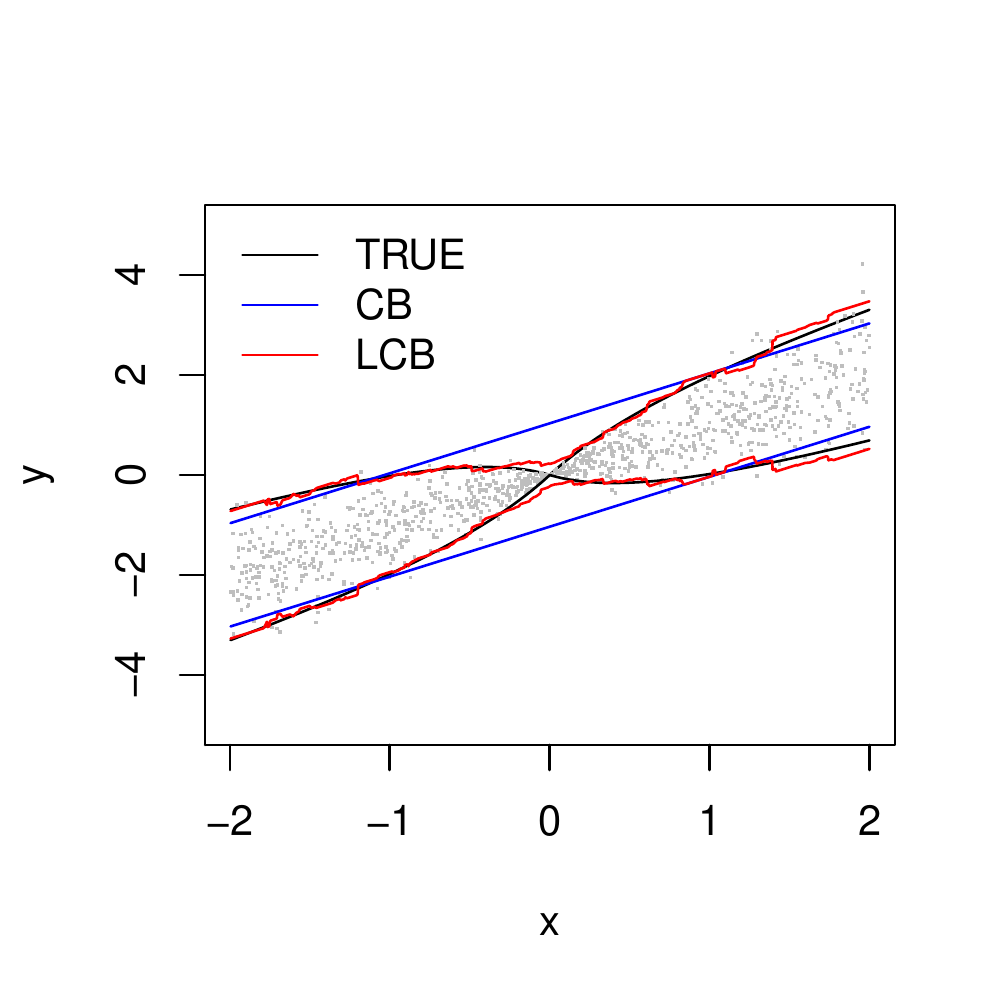} 
\end{center}
\vskip -.2in
\end{figure}
In this paper, we propose a novel approach to build CI using localized conformal inference, which allows for decision rules that may depend on $X_{n+1}$. The main idea is to introduce a localizer around $X_{n+1}$, and up-weight samples close to $X_{n+1}$ according to the localizer. For example, consider a localizer 
\[
H(X_i) =\left\{ \begin{array}{cl} 1& \mbox{if $X_i$ is  among the 100 nearest neighbors of }X_{n+1}\\
0 & \mbox{otherwise}
\end{array}\right..
\]
We  include the response value $y$  in $\hC(X_{n+1})$ if and only if $V(z_{n+1})$ is smaller than the $\tilde\alpha$ quantile of a weighted empirical distribution, where we assign weight $\frac{H(X_i)}{\sum^{n+1}_{j=1} H(X_j)}$ to $V_i$ for $i =1,\ldots, n$ and weight  $\frac{H(X_{n+1})}{\sum^{n+1}_{j=1} H(X_j)}$ to $\infty$. We  show that we can choose $\tilde\alpha$ strategically such that we have finite sample coverage   as described in eq.\ (\ref{eq:goal}).   In Figure \ref{fig:illustration1}, the red curve is the confidence band using the localized conformal inference with the nearest neighbor localizer $H$ that we have just described.  We can see that it does capture the heterogeneity of the underlying truth much better than the conformal confidence band. Performing conformal inference while emphasizing the special role of $X_{n+1}$ is an interesting problem, and to our knowledge, this is the first method providing a theoretical guarantee.

The paper is organized as follows. In Section \ref{sec:related}, we give a brief summary of some related work in conformal prediction with local coverage and weighted conformal prediction. In Section \ref{sec:method}, we  introduce the idea of localized conformal prediction, focusing on the case where we have a fixed score function with i.i.d generated training and test samples. We provide simulation results comparing  localized conformal inference and the conformal inference in Section \ref{sec:sim1}.  In Section \ref{sec:extension}, we give details about how to apply the idea of localized conformal prediction with data-dependent score functions and relate localized conformal inference to the notation of local coverage and asymptotic conditional coverage. Proofs of all Theorems are given in the Supplement 
\section{Related work}
\label{sec:related}
One perspective for capturing the local structure of $V(Z_{n+1})$ at different $X_{n+1}$ is to consider the  conditional coverage validity \cite{lei2014distribution,vovk2012conditional}:
\begin{equation}
\label{eq:CC2}
\bP\{Y_{n+1}\in \hat C(x_0)|X_{n+1} = x_0\} \geq \alpha\;\;\;\;\;\;\;\;\mbox{for all $\mathcal{P}$}.
\end{equation}
However, let $N(\mathcal{P})$ denote a set of non-atom points for $\mathcal{P}$, it is impossible to achieve the finite sample conditional validity without letting $\hat C(x)$ have infinite length for all $x\in N(\mathcal{P})$ \cite{lei2014distribution,vovk2012conditional,barber2019limits}.

Different methods have then been proposed to construct CIs with approximate conditional coverage validity or local coverage validity.  In  \citet{vovk2012conditional}, \citet{lei2014distribution} and \citet{barber2019limits}, the authors  partition the feature space into  $K$ finite subsets and applies conformal inference to each of the subsets:
\[
\bP\{Y_{n+1}\in \hC(X_{n+1})|X_{n+1}\in \mathcal{X}_k\} \geq \alpha, \; \forall k = 1,2,\ldots, K.
\]
for some fixed partition $\cup^K_{k=1} \mathcal{X}_k = \real^p$.  This approach requires to fix  $\cup^K_{k=1} \mathcal{X}_k$  before looking at the test sample $X_{n+1}$. In particular, with $\cup^K_{k=1} \mathcal{X}_k$ being a fixed partition, we can have less than ideal performance for $X_{n+1}$ close to the boundary of $\mathcal{X}_k$.   A second approach is to  reweight the empirical distribution of $\{X_1,\ldots,X_n, X_{n+1}\}$ with $m$ different Gaussian kernels centered at a set of fixed points $\{x_i\in \real^p, i = 1,\ldots, m\}$, and correspondingly, construct $m$ different confidence intervals $\hC(X_{n+1}, x_i)$, $i=1,\ldots, m$ for $Y_{n+1}$. The final CI $\hC(X_{n+1}) = \cup^{m}_{i=1}\hC(X_{n+1}, x_i)$ is the union of all constructed CIs \citep{barber2019conformal}. Similar to the previous approach, it is not ideal to have fixed $\{x_i\in \real^p, i = 1,\ldots, m\}$, and the action of taking the union may lead to unnecessarily wide CIs. Another line of related work consider better score functions for constructing the prediction intervals \citep{lei2018distribution,romano2019conformalized, kivaranovic2019adaptive, sesia2020comparison,izbicki2019distribution,chernozhukov2019distributional}. In \citet{romano2019conformalized} and \citet{kivaranovic2019adaptive}, the authors use score functions based on estimated quantiles instead of  estimated mean values. In \citet{izbicki2019distribution} and \citet{chernozhukov2019distributional}, the authors consider score functions based on estimates of the conditional cumulative distribution\slash density  of $y$ given $x$. Such methods tend to enjoy good empirical results when we can learn the quantiles/densities reasonably well. They  are very different from the localized conformal prediction. The former find pre-fixed score functions $V(.)$ that are more homogenous but follows the usual conformal prediction construction once given the score function, while the later extends the core idea of conformal prediction and allows up-weighting samples close to a given test sample for any score function considered.  They are also not competing with each other: we can always use a perhaps better score function, and apply the idea of localized conformal prediction to guard against overall poor results driven by a small percent of regions, or significant remaining heterogeneity due to a bad model fit. Another slightly related idea is to reweight the training samples to match the distribution of a batch of test samples \cite{barber2019conformal}. In \citet{barber2019conformal}, the authors consider an average coverage for the test sample distribution given that we have enough test samples to estimate the covariate shift \citep{shimodaira2000improving, sugiyama2005input, sugiyama2007covariate,quionero2009dataset}: it uses a pre-fixed weighting function to control the marginal coverage under the test data distribution, and the CIs around the scores will still be homogeneous. Again, it is different from the localized conformal prediction nor does it compete with the localized conformal prediction. In the supplement, we further demonstrate this by applying localized conformal prediction to the covariate shift scenario, and show that the localized\&covariate shift conformal prediction can have narrower confidence bands for regions with enough training samples compared when the distributions of the training and test data are very different.

\section{Localized conformal inference with fixed score function $V(.)$}
\label{sec:method}

We start with the setting where the score function $V(.)$ is fixed. For example,  $V(x, y) = |y - \mu(x)|$ where $\mu(x)$ is a fixed prediction function for the response $y\in \real$ at $x\in \real^p$. In practice, this can correspond to the cases where
\begin{enumerate}
\item  We perform sample splitting, using one fold of the data to train $V(.)$ and the other fold to perform conformal inference.
\item  We have learned $V(.)$ from previous data, but want to apply it to a new data set.
\end{enumerate}

Let  the localizer function $H(x_1, x_{2}, X)\in [0,1]$ for $x_1, x_2 \in \real^p$ be a function that may depend on the  set $X=\{X_1,\ldots, X_{n+1}\}$, and always satisfies $H(x, x, X) = 1$ for all $x\in \real^p$. For the convenience of notation, we define $H_i(.)\coloneqq H(X_i, .,X)$ be the localizer centered at $X_i$, and $H_{i,j} \coloneqq H_i(X_j) = H(X_i, X_j, X)$.  For any distribution $\mathcal{F}$ on $\real$, define its level $\alpha$ quantile as 
\[
Q(\alpha;\mathcal{F}) = \inf\{t: \bP\{T \leq t|T\sim \mathcal{F}\} \geq \alpha\}
\]
Let $\delta_v$  be a  point mass at $v$,  $v_{1:n}\coloneqq \sum^n_{i=1}\delta_{v_i}$ be the empirical distribution of $\{v_1,\ldots, v_n\}$, and $v_{1:n}\cup v_{n+1}\coloneqq \sum^{n+1}_{i=1}\delta_{v_i}$ be the empirical distribution of $\{v_1, \ldots, v_n, v_{n+1}\}$.

The biggest difference between conformal inference and localized conformal inference is that, instead of using the level $\alpha$ quantile of the empirical distribution,  we consider the level $\tilde\alpha$ quantile of the weighted empirical distribution, with weight proportional to $H_{n+1,i}$ for sample $X_i$. The weights allow us to emphasize more the samples close to $X_{n+1}$. Let  $p^{H}_{i,j} \coloneqq \frac{H_{ij}}{\sum^{n+1}_{k=1} H_{ik}}$ for $i,j = 1,\ldots, n+1$, and define $\hat\mF_i \coloneqq  \sum^{n+1}_{j=1}p^H_{i, j}\delta_{V_j}$ as the weighted empirical distribution of $\{V_1,\ldots, V_{n}, V_{n+1}\}$ using the localizer centered at $X_i$, for $i = 1,\ldots, n+1$. Let  $\hat \mF = \sum^n_{i=1} p^H_{n+1,i}\delta_{V_i}+p^{H}_{n+1,n+1}\delta_{\infty}$ be the distribution replacing $V_{n+1}$ with $\infty$ in $\hat \mF_{n+1}$.  We show that $\tilde\alpha$ can be strategically chosen to guarantee the finite sample coverage. 
\begin{corollary}
\label{cor:split1} 
Let $Z_1,\ldots,Z_{n+1}\overset{i.i.d}{\sim}\mathcal{P}$, and $V(.)$ be a fixed function. For any $\tilde\alpha$, let $v^*_{i} =  Q(\tilde\alpha; \hat \mF_i), i = 1,2,\ldots, n+1$. If $\tilde\alpha$ satisfies 
\begin{equation}
\label{eq:goal1}
\sum^{n+1}_{i=1}\frac{1}{n+1}\mathbbm{1}_{V_i \leq v^*_i}  \geq \alpha. \tag{G1}
\end{equation}
then $\bP\left\{V_{n+1} \leq Q(\tilde\alpha;\hat \mF_{n+1} ) \right\}\geq \alpha$, and thus,  $\bP\left\{V_{n+1} \leq Q(\tilde\alpha;\hat \mF) \right\}\geq \alpha$.
\end{corollary}
Corollary \ref{cor:split1} is a special case of Theorem \ref{thm:general1}.   Here, we provide some intuition for why such an $\tilde\alpha$ can guarantee a level $\alpha$ coverage. Conformal inference relies on the exchangeability of data. However, when weighting samples based on a localizer, we break the exchangeability in the training and test samples. Corollary \ref{cor:split1} suggests a way of picking $\tilde\alpha$ that restores some underlying exchangeability, by considering not only the weighted samples based on the localizer around $X_{n+1}$, but also localizers around each of the trainings samples $X_1,\ldots, X_n$. 

It is obvious that for $V_{n+1} = V(X_{n+1},y)$ and any given $y$, $v^*_i$ is non-decreasing in $\tilde\alpha$. Thus, $\tilde\alpha$ satisfies eq.\ (\ref{eq:goal1}) if any smaller value satisfies it. In practice, we would like to pick a small $\tilde\alpha$ in order to construct a short CI. Based on Corollary \ref{cor:split1}, to obtain an interval $\hC(X_{n+1})$ for $Y_{n+1}$, for every possible response value $y$, we let $\tilde\alpha(y)$ be the smallest value for $\tilde\alpha$ that we can find such that eq.\ (\ref{eq:goal1}) holds with $V_{n+1} = V(X_{n+1},y)$, and include $y$ in $\hC(X_{n+1})$ if $V(X_{n+1},y) \leq Q(\tilde\alpha(y);\hat \mF)$.

Such an algorithm is too computationally expensive to carry out in practice.  We instead provide Theorem \ref{thm:split2} which is the foundation of a practical procedure.  
\begin{theorem}
\label{thm:split2}
Let  $Z_1,\ldots,Z_{n+1}\overset{i.i.d}{\sim}\mathcal{P}$, and $V(.)$  be a fixed function. For any $\tilde\alpha$, let $\bar v^* =  Q(\tilde\alpha; \hat \mF)$, $v^*_{i1} =Q(\tilde\alpha; \sum^n_{j=1}p^{H}_{i,j} \delta_{V_j}+p^{H}_{i,n+1}\delta_{\bar v^*}) $,  $v^*_{i2} = Q(\tilde\alpha;\sum^n_{j=1}p^{H}_{i,j} \delta_{V_j}+ p^{H}_{i,n+1}\delta_0)$. If $\bar v^* = \infty$ or if
\begin{equation}
\label{eq:goal2}
 \sum^{n}_{i=1}\frac{1}{n+1}\mathbbm{1}_{V_i \leq v^*_{i1}} \geq \alpha \quad\mbox{and}\quad\sum^{n}_{i=1}\frac{1}{n+1}\mathbbm{1}_{V_i \leq v^*_{i2}} +\frac{1}{n+1}\geq \alpha.  \tag{G2}
\end{equation}
then we have  $\bP\left\{V_{n+1} \leq Q(\tilde\alpha;\hat \mF ) \right\}\geq \alpha$.
\end{theorem}
Instead of finding the value  $\tilde\alpha(y)$ that makes eq.\ (\ref{eq:goal1}) hold for each $y$, we find $\tilde\alpha$ that makes eq.\ (\ref{eq:goal1}) hold for all $y$ simultaneously.  Two components in eq.\ (\ref{eq:goal2}) can be viewed as two hardest cases in eq.\ (\ref{eq:goal1}): (1)  when $V(X_{n+1},y)= 0$, and (2)  when  $V(X_{n+1},y)=\bar v^*$.  This argument can provide some intuition, although the actual proof is more complicated than this. Based on  Theorem \ref{thm:split2}, we can use Algorithm \ref{alg:alg1} to construct the CI for $Y_{n+1}$, which first constructs the CI for $V_{n+1}$ by doing a grid search over  a set of candidate values of $\tilde\alpha$ to  find a small value satisfying eq.\ (\ref{eq:goal1}).
\begin{algorithm}[H]
\caption{Localized conformal inference with fixed score function}
\label{alg:alg1}
\hspace*{\algorithmicindent} \textbf{Input:} Level $\alpha$, scores $V$, weights matrix $p^{H}$ and  grid values $0\leq \alpha_1<\ldots < \alpha_M\leq 1$ (for $\tilde\alpha$).\\
\hspace*{\algorithmicindent} \textbf{Output:} The constructed CI as $\hC(X_{n+1}) = \{y: V(X_{n+1}, y) \leq Q(\tilde\alpha; \hat \mF)\}$.\\
\begin{algorithmic} 
\STATE 1. Grid-search for $\tilde\alpha$, find the smallest value such that either eq.\ (\ref{eq:goal2}) holds or $Q(\tilde\alpha; \hat \mF) = \infty$.
\STATE 2. Invert $V(X_{n+1}, y) \leq Q(\tilde\alpha; \hat \mF)$ to construct the CI for $Y_{n+1}$.
\end{algorithmic}
\end{algorithm}
Now, $\tilde\alpha$ and $\hat \mF$ do not depend on $y$, and typically, it is easy to invert $V(x, y) \leq Q(\tilde\alpha; \hat \mF)$ for any given $x$. As a direct application of Theorem \ref{thm:split2}, Algorithm \ref{alg:alg1} achieves the finite sample coverage guarantee. 
\begin{corollary}
\label{thm:splitC}
Let $Z_1,\ldots,Z_{n+1}\overset{i.i.d}{\sim}\mathcal{P}$, and $V(.)$ be a fixed function. Let $\hC(X_{n+1})\coloneqq  \{y: V(X_{n+1}, y) \leq Q(\tilde\alpha,\hat \mF)\}$  as described in Algorithm \ref{alg:alg1}. Then we have $\bP\{Y_{n+1}\in \hC(X_{n+1})\} \geq \alpha$.
\end{corollary}

Usual conformal inference is a special case of localized conformal inference when $H_{i,j} = 1$. 
\begin{proposition}
\label{prop:equivalence}
Let  $H_{i,j} = 1$, $\forall i,j =1,\ldots,n+1$, and let $\tilde\alpha =\alpha$. Then, either $\bar v^* = \infty$ or eq.\;(\ref{eq:goal2}) holds, and Theorem \ref{thm:split2} recovers the result that $\bP\left\{V_{n+1} \leq Q(\alpha; V_{1:n}\cup \{\infty\} ) \right\}\geq \alpha$.
\end{proposition}
Two questions the reader may want to ask are: (1) how tight is the coverage of the localized conformal prediction CI, and (2) what happens if we simply let $\tilde\alpha = \alpha$ without tuning it based on eq.\ (\ref{eq:goal2})? The answer to both of these will depend on the localizer $H$.   In Theorem \ref{cor:split1}, the coverage may not be exactly $\alpha$ because we may not be able to select  $\tilde\alpha$ such that $\sum^{n+1}_{i=1}\frac{1}{n+1}\mathbbm{1}_{V_i \leq v^*_i}  = \alpha$ exactly.  However, Corollary \ref{cor:split3} says that if we take a random decision rule to get rid of the rounding issue in Corollary \ref{cor:split1}, then the resulting randomized decision rule will be tight. 
\begin{corollary}
\label{cor:split3}
In the setting of Theorem \ref{thm:split2}, for any $\alpha \in (0,1)$, let $\tilde\alpha_1$ be the smallest value of $\tilde\alpha$ such that $ \sum^{n+1}_{i=1}\frac{1}{n+1}\mathbbm{1}_{V_i \leq v^*_i}  \geq \alpha$,  and let $\tilde\alpha_2$ be the largest of $\tilde\alpha$ such that $\sum^{n+1}_{i=1}\frac{1}{n+1}\mathbbm{1}_{V_i \leq v^*_i}  < \alpha$. Let $\alpha_1$, $\alpha_2$ be the values of  $ \sum^{n+1}_{i=1}\frac{1}{n+1}\mathbbm{1}_{V_i \leq v^*_i}$  attained at $\tilde{\alpha}_1$, $\tilde\alpha_2$, and let $\tilde\alpha = \left\{\begin{array}{cc}\tilde \alpha_1&w.p.\;\frac{\alpha - \alpha_2}{\alpha_1-\alpha_2}\\
\tilde\alpha_2 & w.p.\;\frac{\alpha_1 - \alpha}{\alpha_1-\alpha_2}
\end{array}\right.$. Then, we have $\bP\left\{V_{n+1} \leq Q(\tilde\alpha;\hat \mF ) \right\}= \alpha$.
\end{corollary}
Corollary \ref{cor:split3} is a special case of Theorem \ref{thm:general2}.  For the second question, we provide  Example \ref{exm:counter1} and Example \ref{exm:counter2} here, which show that letting $\tilde\alpha = \alpha$ may lead to both over-coverage and under-coverage.
\begin{example}
\label{exm:counter1}
Let $\alpha\in (0,1)$, let $\mathcal{P}_X$ be any  jointly continuous density for feature $x$,  and  consider the localizer $H(x_1, x_2) =\exp(-\frac{|x_1-x_2|}{\sigma})$. For any $1>\epsilon > \alpha$, we can always choose $\sigma$ to be small enough such that with probability at least $\epsilon$, we have $\sum^{n}_{i=1} H(X_{n+1}, X_i) < \frac{1}{\alpha}$ when $X_1, \ldots, X_{n+1}$ are independently generated from $\mathcal{P}_X$. Then, with probability at least $\epsilon$, we will have $Q(\alpha; \hat F) = \infty > V_{n+1}$. Hence, the achieved coverage is at least $\epsilon > \alpha$.
\end{example}

\begin{example}
\label{exm:counter2}
We consider an intuitive approach that practitioners may want to perform in practice:  Let $H_{i,j} = \mathbbm{1}_{|X_j - X_i|\leq h}$  for some fixed distance $h$ and  let $\tilde\alpha = \alpha\in (0,1)$.  Consider the following distribution: 
\[
X_i= \left\{\begin{array}{c c}-1 & \mbox{w.p}\; \frac{1-\alpha}{2-\alpha}\\ 0 & \mbox{w.p}\; (1-\frac{2(1-\alpha)}{2-\alpha}) \\ 1 & \mbox{w.p}\;  \frac{1-\alpha}{2-\alpha} \end{array}\right.
\]
and $Y_i = X_i +\epsilon_i$ with $\epsilon_i|X_i \sim \mbox{Uniform}([-2|X_i|, 2|X_i|])$. Let $V(x,y)=|y-x|$.  Then $V_i \sim \mbox{Uniform}(0, 2|X_i|)$. Suppose we set $h = 1.5$, and consider the asymptotic case when $n\rightarrow \infty$: If $X_{n+1} = 1$, we know that the method considers only training samples at $1$ and at $0$, with asymptotic proportions $(1-\alpha)$ and $\alpha$ respectively.  Then $Q(\alpha; \hat \mF)\rightarrow 0$ at $X_{n+1} = 1$ and $P(V_{n+1} \leq Q(\alpha; \hat \mF)|X_{n+1} = 1)\rightarrow 0$. Similarly, we have $P(V_{n+1} \leq Q(\alpha; \hat \mF)|X_{n+1} = -1)\rightarrow 0$. Thus, the achieved coverage is asymptotically $1-\frac{2(1-\alpha)}{2-\alpha}$, and we have an under-coverage of
\[
\alpha - (1-\frac{2(1-\alpha)}{2-\alpha}) = \frac{\alpha(1-\alpha)}{2-\alpha}, \forall \alpha \in [0,1]
\]
\end{example}

\subsection{Choice of H}
\label{sec:method_h}
The choice of $H$ will greatly influence how localized our algorithm is.  Suppose that we have a data set $\mathcal{D}_0$ which is generated according to $\mathcal{P}$ and is  independent of $Z=\{Z_1,\ldots, Z_n, Z_{n+1}\}$. We  consider two types of localizers and will tune them using $\mathcal{D}_0$:
 \begin{enumerate}
 \item  Distance based localizer: $H_h(x_1, x_2,X)  = \mathbbm{1}_{\{|\frac{x_2- x_1}{h}| \leq 1\}}$.
 \item Nearest-neighbor based localizer: $H_h(x_1, x_2,X) = \mathbbm{1}_{\{|x_1-x_2| \leq Q(\frac{h}{n}; \sum^{n+1}_{i=1}\delta_{|X_i - x_1|}) \}}$.
 \end{enumerate}
 In practice, we want to  to tradeoff between locality and volatility, and choose $h$ to have relatively narrow and stable CIs for most of the samples. We propose a way of doing it in a data adaptive way, and we give details of the proposal in Supplement section \ref{app:choiceH}. Also, when the dimension of the features is high, we may want to find some low dimensional space to capture the heterogeneity in the score functions and use weights based on the low dimensional projected distances. Finding good low dimensional projection with high dimensional data is a non-trivial and separate topic, and we consider only low dimensional features in the main paper. A  relatively simple high dimensional example is given in Supplement section \ref{app:choiceH} for illustrating purpose.
\section{Empirical study with fixed V(.)}
\label{sec:sim1}
We compare the localized conformal inference band (LCB) using Algorithm \ref{alg:alg1} and  conformal inference band (CB) in this section.  
\begin{example}
\label{exm1}
Let $X_i \sim N(0,1)$ and $Y_i= X_i+\epsilon_i$ for $i = 1,\ldots, n+1$.  We use the fixed score function $V(X_i, Y_i) = |Y_i - X_i|$ to do inference for both the conformal and localized conformal approaches.  For localized conformal inference, we consider the distance based localizer  $H^1_h(X_j, X_i) = \mathbbm{1}_{|X_i - X_j|\leq h}$ and  the nearest-neighbor based localizer $H^2_h(X_j,X_i) = \mathbbm{1}_{|X_i - X_j|\leq Q(\frac{h}{n}; \sum^{n+1}_{k=1}\frac{1}{n+1}\delta_{|X_k - X_i|})}$ for $h$ nearest neighbors. We try three different values  $h_1$, $h_2$ and $h_3$ of the tuning parameter $h$. For $H_h^1$, we let $h_1 = .1$, $h_2 = 1$ and $h_3 = \hat h_1$, and for $H^2_h$, we let $h_1 = 40$, $h_2 = 500$ and $h_3 = \hat h_2$, where $\hat h_1$ and $\hat h_2$ are  automatically chosen using another $i.i.d$ generated data set with $n$ samples according to Appendix \ref{app:choiceH} in the Supplement.

For each of the following noise generating mechanisms, we let $n = 500$ and repeat the experiment 1000 times: (a) $\epsilon_i \overset{i.i.d}{\sim} N(0,1)$, (b)$\epsilon_i|X_i \sim \frac{1}{2|X_i|+1}N(0,1)$,  or (c)$\epsilon_i|X_i \sim \frac{|X_i|}{|X_i|+1}N(0,1)$.   Table \ref{tab:exm1} shows the achieved coverage for $\alpha = .80$ and $\alpha = 0.95$. We can see that both conformal prediction and localized conformal prediction with different localizers have achieved the desired coverage. Figure \ref{fig:exm1} shows the constructed confidence bands across 1000 repetitions using different methods at $\alpha = .95$.

As $h$ increases, the localized conformal bands become more similar to the conformal bands.  Comparing results for  localized conformal inference with $h = h_1$ and $h = h_2$, we see that  small $h$ reveals more local structure.  Using the automatic tuning procedure, we have successfully chosen large $h$ when the underlying distribution of $V(X_{n+1})$ is homogeneous and small $h$ when it is heterogeneous across different values of $X_{n+1}$.
\begin{table}
\centering
\caption{Example \ref{exm1}, Coverage. Column names $h_1, h_2, h_3$ represent the tuning parameter being 0.1, 1, $\hat h_1$ for the distance based localizer $H^1_h$ and tuning parameters being  $40, 500, \hat h_2$ for the nearest-neighbor based localizer $H_h^2$.}
\label{tab:exm1}
\begin{tabular}{|l|rrr|rrr|rrr|}
  \hline
    $\alpha = .95$ & &(a) &  & &(b) &   & &  (c)  & \\ 
     \hline
 & $h_1$ &$h_2$  &$h_3$  & $h_1$ &$h_2$ &$h_3$ & $h_1$ &$h_2$&$h_3$    \\ 
  \hline
CB &0.95 & &  & 0.94 &  &  & 0.96 &  &  \\ 
  LCB, $H_1$ &0.96 & 0.96 & 0.96 & 0.96 & 0.94 & 0.95 & 0.96 & 0.96 & 0.96 \\ 
    LCB, $H_2$ & 0.96 & 0.96 & 0.96 & 0.96 & 0.94 & 0.95 & 0.96 & 0.96 & 0.96 \\ 
   \hline
     $\alpha = .80$ & $h_1$ &$h_2$  &$h_3$  & $h_1$ &$h_2$ &$h_3$ & $h_1$ &$h_2$&$h_3$ \\ 
     \hline
CB &  0.80 &  &  & 0.81 &  &  & 0.81 &  &  \\ 
  LCB, $H_1$  &  0.81 & 0.81 & 0.81 & 0.82 & 0.82 & 0.82 & 0.81 & 0.82 & 0.81 \\ 
    LCB, $H_2$  &0.81 & 0.81 & 0.81 & 0.82 & 0.82 & 0.82 & 0.81 & 0.82 & 0.81 \\ 
   \hline
\end{tabular}
\vskip -.1in
\end{table}
\begin{figure}
\caption{\em Example \ref{exm1}. Confidence bands constructed using 1000 repetitions with targeted level at $\alpha = .95$. The black,  blue, red and green dots respectively represent (1) the true responses for the test samples (response), (2) the conformal confidence bands (CB), (3) the localized conformal confidence bands with distance localizer $H_h^1$ (LCB1), and (4) the localized conformal confidence bands with nearest-neighbor based localizer $H_h^2$ (LCB2). The red dots close to the top and bottom within each plot represent samples whose CIs based on LCB1 have infinite length (both the CB and the LCB2 do not have infinite length CI by construction).}
\label{fig:exm1}
\begin{center}
\includegraphics[width = .8\textwidth, height = .5\textwidth]{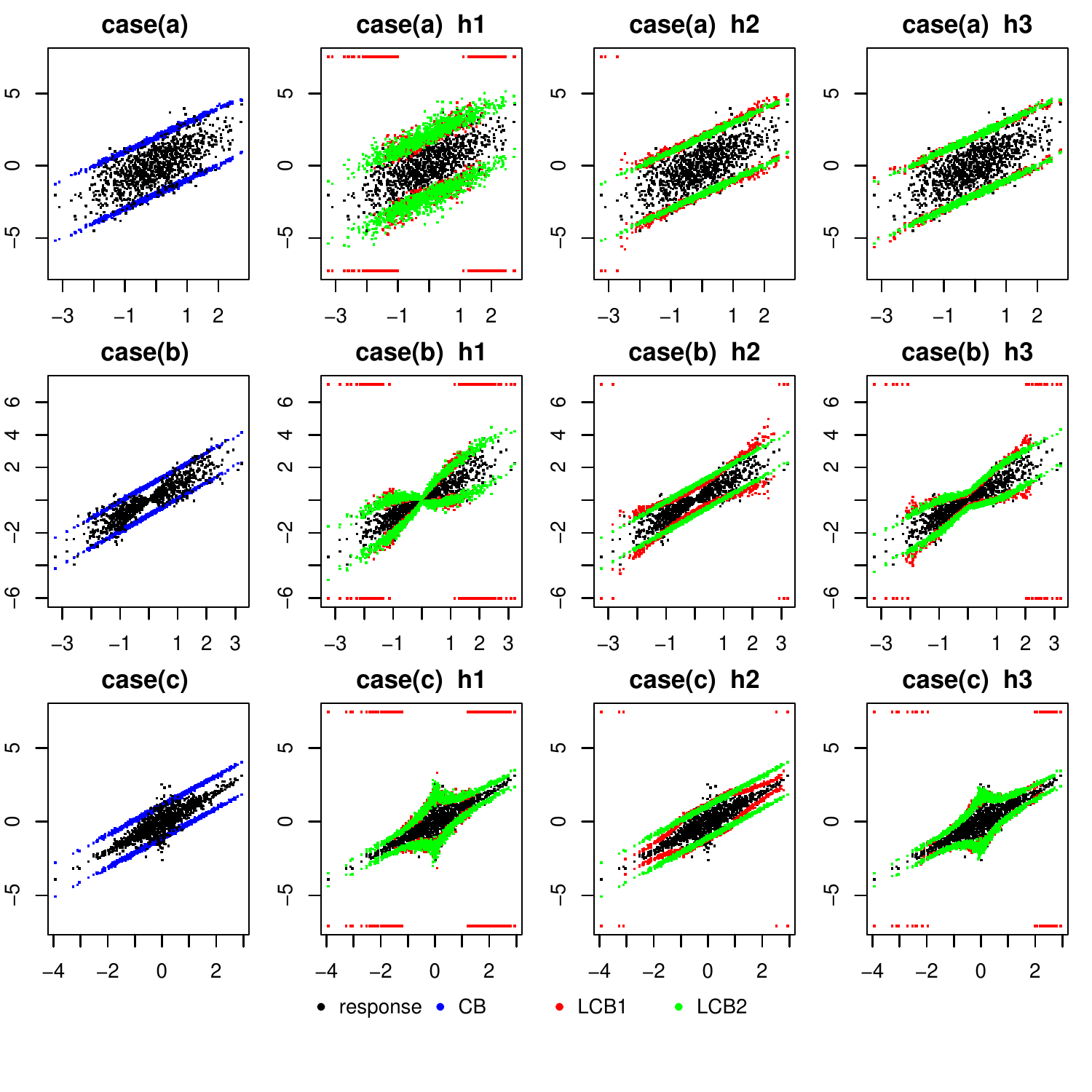} 
\end{center}
\vskip -.4in
\end{figure}
\end{example}

\section{Extensions}
\label{sec:extension}
In this part, we  draw a connection between the method of localized conformal and the notion of local coverage/asymptotic conditional coverage, and consider an extension where the score function can be data-dependent (but the exchangeability requirement is still needed). 

\subsection{Local and asymptotic conditional coverage}
In  \cite{barber2019conformal}, the authors suggest to consider the following type of local coverage: let $x_0\in \real^p$ and $P^{x_0}_X$ be a distribution concentrated at $x_0$, with $\frac{d P^{x_0}_X(x)}{d x} \propto \frac{d P_X(x)}{d x}K(\frac{x- x_0}{h})$ and $K(\frac{x}{h})$ being the Gaussian kernel with bandwidth $h$, we would like $\hC(x_0)$ such that 
\begin{equation}
\label{eq:local0}
\bP\{\tilde Y_{n+1}\in \hC(x_0)\}\geq \alpha,\; \tilde X_{n+1}\sim P^{x_0}_X, \tilde Y_{n+1} |\tilde X_{n+1} \sim P_{Y|X}
\end{equation}
 The proposal discussed in \cite{barber2019conformal} considers the situation where we want eq.\ (\ref{eq:local0}) to hold for a  set of fixed values for $x_0$, and different values of $x_0$ can lead to different constructed CIs for   a new observation $X_{n+1}$ (see section \ref{sec:related}).

With the score function $V(.)$  fixed, localized conformal inference provide a simple way  to construct a unique $\hC(X_{n+1})$ for every test sample such that  the type local coverage requirement defined in eq.(\ref{eq:local0}) is satisfied.

Let $H(x_1, x_2, X) = H(x_1, x_2)$ to be a data-independent localizer and define the localized distribution $P^{x_0}_X(.)$ around $x_0$ as $\frac{d P^{x_0}_X(x)}{d x} \propto \frac{d P_X(x)}{d x}H(x_0, x)$, and $\hat \mF$ is defined with the localizer $H(X_{n+1}, .)$.
\begin{theorem}
\label{thm:approximate}
\textbf{(a)} Let $Z_1,\ldots, Z_{n}\overset{i.i.d}{\sim} \mathcal{P}$ and  $V(.)$ to be fixed.  Let $\hC(X_{n+1})\coloneqq  \{y: V(X_{n+1}, y) \leq Q(\alpha,\hat \mF)\}$. Conditional on $X_{n+1}$, let $\tilde X_{n+1}|X_{n+1} \sim P^{X_{n+1}}_X$, and $\tilde Y_{n+1}|\tilde X_{n+1}\sim P_{Y|X}$. Then, we have  $\bP\{\tilde Y_{n+1} \in \hC( X_{n+1})|X_{n+1} = x_0\}\geq \alpha$  for all  $x_0$. 

\textbf{(b)} If the data distribution satisfies regularity conditions: (1) $X$ is on $[0,1]^d$ with marginal density  satisfying $0<b_1\leq p_X(x)\leq b2 <\infty$ for constants $b_1$, $b_2$. (2) The conditional density of $V$ given $X$ is Lipschitz in $X$: $\|p_{V|X}(.|x) - p_{V|X}(.|x')\|_{\infty} \leq L\|x-x'\|$ for a constant $L$. Then, we have $[\alpha - \bP(Y_{n+1}\in \hC(X_{n+1})|X_{n+1}=x_0)]_+\overset{h\rightarrow0}{\rightarrow} 0$.
\end{theorem}
 The confidence interval $\hat C(X_{n+1})$ is indexed by $X_{n+1}$, and when the training set does not change, we will have a unique confidence interval for every realization of $X_{n+1}$. The regularity conditions here are essentially a subset of regularity conditions used in \citet{lei2014distribution}. With a Gaussian kernel, we can have local coverage when setting $\tilde\alpha = \alpha$ under relatively mild regularity conditions as the bandwidth $h\rightarrow 0$. We defer the discussion of efficiency as a future work since it will depend on both the score function used, sample size and localizers, while we want to focus on introducing the idea of localized conformal prediction as an extension to the general conformal prediction framework.

\subsection{Localized conformal inference with data-dependent score function}
\label{sec:method1}
In this section, we consider a more general case where the score function can have some data dependency but still leads to exchangeability. Let $Z = \{Z_1,\ldots, Z_n, Z_{n+1}\}$ be the set of  training  and test samples, the score function can depend on the set $Z$ but not their ordering, and have the form $V(., Z)$.  To accommodate for this more general case and distinguish it from the case with fixed score function, we introduce some new notations for convenience. Define $V^{z_{n+1}}_i \coloneqq V(Z_i, Z)|_{Z_{n+1} = z_{n+1}}$ and $\hat \mF^{z_{n+1}}_i \coloneqq \left(\sum^{n+1}_{j=1} p^H_{i,j}\delta_{V^{Z_{n+1}}_i}\right)|_{Z_{n+1} = z_{n+1}}, \forall i = 1,\ldots, n+1$, and $\hat \mF^{z_{n+1}} \coloneqq \left(\sum^{n}_{j=1} p^H_{n+1,j}\delta_{V^{Z_{n+1}}_j}+p^H_{n+1,n+1}\delta_\infty\right)|_{Z_{n+1} = z_{n+1}}$, as the realizations of $V(Z_i, Z)$ and the weighted distribution $\hat \mF_i$ , $\hat \mF$ at $Z_{n+1} = z_{n+1}$. For example, $V^{z_{n+1}}_{n+1} = V(z_{n+1}, \{Z_1,\ldots, Z_n, z_{n+1}\})$ and  $V^{z_{n+1}}_{i} = V(Z_i, \{Z_1,\ldots, Z_n, z_{n+1}\})$ for $i = 1,\ldots, n$. We will always use $V$ and $\mF$ with the superscript to represent that data-dependency is allowed, and use the ones without superscript to represent that the score function is fixed.  Theorem \ref{thm:general1} and   Theorem \ref{thm:general2}  are extensions of Corollary \ref{cor:split1} and \ref{cor:split3}, allowing for data dependent score functions.
\begin{theorem}
\label{thm:general1} 
Let $Z_1,\ldots,Z_{n+1}\overset{i.i.d}{\sim}\mathcal{P}$. For any $\tilde\alpha$, define  $v^*_{i} =  Q(\tilde\alpha; \hat \mF^{z_{n+1}}_i)$, $i = 1,2,\ldots, n+1$. If $\tilde\alpha$ satisfies
\begin{equation}
\label{eq:goal3}
\sum^{n+1}_{i=1}\frac{1}{n+1}\mathbbm{1}_{V^{Z_{n+1}}_i \leq v^*_i}  \geq \alpha  
\end{equation}
Then $\bP\left\{V^{Z_{n+1}}_{n+1} \leq Q(\tilde\alpha;\hat \mF^{Z_{n+1}}_{n+1} ) \right\}\geq \alpha$, and thus,  $\bP\left\{V^{Z_{n+1}}_{n+1} \leq Q(\tilde\alpha;\hat \mF^{Z_{n+1}}) \right\}\geq \alpha$.
\end{theorem} 
\begin{remark}
\label{remark1}
Same as in Proposition \ref{prop:equivalence}, when  $H_{i,j} = 1$, we have $\hat \mF^{Z_{n+1}} = V^{Z_{n+1}}_{1:n}\cup \{\infty\}$, and $v^*_i = Q(\tilde\alpha; V^{Z_{n_1}}_{1:(n+1)})$, $\forall i = 1,\ldots, n+1$. Since eq.(\ref{eq:goal3}) holds for  if and only if $v^*_i \geq Q(\alpha; V^{Z_{n_1}}_{1:(n+1)})$. We recovered the conformal inference result \citep{vovk2005algorithmic}: $\bP\left\{V^{Z_{n+1}}_{n+1} \leq Q(\alpha; V^{Z_{n+1}}_{1:n}\cup \{\infty\} ) \right\}\geq \alpha$.
\end{remark}
Corollary \ref{cor:generalC} is a direct application of  Theorem\ref{thm:general1}.
\begin{corollary}
\label{cor:generalC} 
In the setting of Theorem \ref{thm:general1}, let $ z_{n+1} = (X_{n+1}, y)$, and let $\tilde\alpha(y)$ be values indexed by $y$. Let $\hC(X_{n+1}) \coloneqq \{y:V^{z_{n+1}}_{n+1} \leq Q(\tilde\alpha(y);\hat \mF^{z_{n+1}})\}$. If $\tilde\alpha(y)$ satisfies eq.\ (\ref{eq:goal3}) at $Z_{n+1} =z_{n+1}$, we have $\bP\left\{Y_{n+1}\in  \hC(X_{n+1})\right\}\geq \alpha$.
\end{corollary}
Same as in the setting with fixed score function, Theorem \ref{thm:general2} says that if we take a random decision rule to get rid of the rounding issue in Theorem \ref{thm:general1}, then the resulting randomized decision rule will be tight.
\begin{theorem}
\label{thm:general2} 
In the setting of Theorem \ref{thm:general1}, for any $\alpha \in (0,1)$, let $\tilde\alpha_1$ be the smallest value of $\tilde\alpha$ such that $ \sum^{n+1}_{i=1}\frac{1}{n+1}\mathbbm{1}_{V^{Z_{n+1}}_i \leq v^*_i}  \geq \alpha$,  and let $\tilde\alpha_2$ be the largest of $\tilde\alpha$ such that $\sum^{n+1}_{i=1}\frac{1}{n+1}\mathbbm{1}_{V^{Z_{n+1}}_i \leq v^*_i}  < \alpha$. Let $\alpha_1$, $\alpha_2$ be the values of  $ \sum^{n+1}_{i=1}\frac{1}{n+1}\mathbbm{1}_{V^{Z_{n+1}}_i \leq v^*_i}$  attained at $\tilde{\alpha}_1$, $\tilde\alpha_2$, and let $\tilde\alpha = \left\{\begin{array}{cc}\tilde \alpha_1&w.p.\;\frac{\alpha - \alpha_2}{\alpha_1-\alpha_2}\\
\tilde\alpha_2 & w.p.\;\frac{\alpha_1 - \alpha}{\alpha_1-\alpha_2}
\end{array}\right.$. Then, we have $\bP\left\{V_{n+1} \leq Q(\tilde\alpha;\hat \mF ) \right\}= \alpha$.
\end{theorem}

We do not show experiments with the data-dependent score function because this general recipe described in Theorem \ref{thm:general1} is too computationally expensive: for every $y$, we need to retrain our prediction model to get $V^{z_{n+1}}(x)$ and then re-calculate $v^*_i$ and $\tilde \alpha(y)$.  Similar problems are  encountered by the conformal prediction for data dependent score function $V(., Z)$. We include the general setting here for the sake of completeness and show that the idea of localized conformal inference can be extended to the regime where the idea of  conformal prediction also works.

\section{Discussion}
In this paper, we have described a new way perform conformal prediction with localization. This localized conformal prediction approach allows us to focus on a local region around a given test  sample, and have finite sample coverage guarantee without distributional assumptions on $Y|X$.  Besides constructing more localized assumption-free confidence intervals, another interesting application would be to apply the idea of localized conformal prediction under the presence of outliers. Conformal inference has been used in classification problems  for outlier detection \cite{hechtlinger2018cautious,guan2019prediction}. Localized conformal inference with distance-based localizer seems to be an useful framework for making prediction in the presence of outliers, for both regression and classification problems if we can find a proper localizer. Compared with most other outlier detection approaches \cite{hodge2004survey,chandola2009anomaly}, it can use information from the response, since the degree of localization will depend on the distribution of  prediction errors.

\appendix
In Appendix \ref{app:covariateShift}, we generalize the localized conformal prediction to the setting where there might be covariate shift. We compare localized\&covariate shift conformal prediction with the conformal prediction construction under covariate shift\cite{barber2019conformal} and show that the localized version could lead to more desirable construction when the test samples are very different from the training samples. In Appendix \ref{app:proofs}, we provide proofs to all Theorems\slash Propositions in the main paper and in this Supplement. In Appendix \ref{app:choiceH}, we provide details of  picking the ``bandwidth"  $h$ adaptively and automatically for the localizer to achieve narrow and stable confidence bands, and a demonstrative example of applying the localized conformal prediction to high dimensional data.

\section{Localized conformal prediction under covariate shift }
\label{app:covariateShift}
When there is potential covariate-shift, we assume the training and test data can be generated from different distributions in their feature space \citep{shimodaira2000improving, sugiyama2005input, sugiyama2007covariate,quionero2009dataset}:
\begin{align*}
Z_{n+1} \sim \tilde P= \tilde P_X \times P_{Y|X},\;Z_i \overset{i.i.d}{\sim} P = P_X\times P_{Y|X}, i = 1,\ldots, n.
\end{align*}
The distribution of $Y|X$ is still assumed to be the same for the training and test samples. The work of \cite{barber2019conformal} extends conformal inference to this setting. Assuming that $\tilde P_X$ is absolutely continuous with respect to $P_X$, with known $w(x) = \frac{d P_X}{d \tilde P_X}$, we can perform  conformal inference using weighted exchangeability. 
\begin{proposition}[\citet{barber2019conformal}]
\label{prop:prop2}
Let $p_i = \frac{w(X_i)}{\sum^{n+1}_{i=1} w(X_i)}$. For any $\alpha$, we have
\[
P(V^{Z_{n+1}}_{n+1} \leq Q(\alpha; \sum^n_{i=1}p_i\delta_{V^{Z_{n+1}}_i}+ p_{n+1}\delta_{\infty} \}) \geq\alpha.
\]
\end{proposition}
Knowing the density ratio function $w(x)$, we can generalize localized conformal inference to take into consideration the covariate shift in a straightforward manner: both Theorem \ref{thm:localQuantile} and Theorem \ref{thm:localQuantileFix} consider this general case.  
\begin{assumptions}
\label{ass:ass1}
The samples are independently generated and the distributions of the training samples and the test sample can be different due to covariance-shift:
\begin{align*}
&Z_{n+1}\sim \tilde P = \tilde P_X\times P_{Y|X},\;\;Z_i \sim P = P_X\times P_{Y|X},\;\forall i = 1,2,\ldots, n\\
\end{align*}
\end{assumptions}
\begin{assumptions}
\label{ass:ass2}
$\tilde P_X$ is absolute continuous with respect to $P_X$, with  $w(x) = \frac{d \tilde P_X(x)}{d P_X(x)}$.
\end{assumptions}
We consider $w(x)$ to be known. When $w(x) = 1$, $\forall x\in \real^p$, we return to the i.i.d data setting.
\begin{theorem} 
\label{thm:localQuantile}
Suppose that Assumptions \ref{ass:ass1} -  \ref{ass:ass2}  hold. For any  $\tilde\alpha$, define  $v^*_{i} =  Q(\tilde\alpha; \hat \mF^{Z_{n+1}}_i), i = 1,\ldots, n+1$. If 
\begin{equation}
\label{eq:goal1w}
\sum^{n+1}_{i=1}\frac{w(X_i)}{\sum^{n+1}_{j=1} w(X_j)}\mathbbm{1}_{V^{Z_{n+1}}_i \leq v^*_i}  \geq \alpha  \tag{$G1^w$}
\end{equation}
Then  $P(V^{Z_{n+1}}_{n+1} \leq  Q(\tilde\alpha;\hat \mF_{n+1}^{Z_{n+1}}) \geq \alpha$, and thus,  $P(V^{Z_{n+1}}_{n+1} \leq  Q(\tilde\alpha;\hat \mF^{Z_{n+1}}) \geq \alpha$.
\end{theorem}
Corollary \ref{cor:invert}  is a direct application of Theorem \ref{thm:localQuantile}.
\begin{corollary}
\label{cor:invert} 
In the setting of Theorem \ref{thm:localQuantile}, let $z_{n+1} = (X_{n+1}, y)$, and  $\tilde\alpha(y)$ be any value that satisfies eq.\ (\ref{eq:goal1w}) when $V_{n+1} = V(z_{n+1})$. Let  $\hC(X_{n+1}) \coloneqq \{y:V^{z_{n+1}}_{n+1} \leq Q(\tilde\alpha(y); \hat \mF^{z_{n+1}} )\}$. Then $P(y\in \hC(X_{n+1}) ) \geq \alpha$.
\end{corollary}
Theorem \ref{thm:localQuantile} provides a way to choose $\tilde\alpha$ with guaranteed coverage, by considering localizers centered at  each of the samples to restore exchangeability. Theorem \ref{thm:tightness} says that if we take a random decision rule to get rid of the rounding issue, we can have an algorithm with tight coverage.
\begin{theorem} 
In the setting of Theorem \ref{thm:localQuantile}, for any $\alpha\in (0,1)$, let $\tilde\alpha_1$ be the smallest value of $\tilde\alpha$ such that $ \sum^{n+1}_{i=1}\frac{w(X_i)}{\sum^{n+1}_{j=1} w(X_j)}\mathbbm{1}_{V^{Z_{n+1}}_i \leq v^*_i}  \geq \alpha$,  and  let $\tilde\alpha_2$ be the largest value of $\tilde\alpha$ such that $\sum^{n+1}_{i=1}\frac{w(X_i)}{\sum^{n+1}_{j=1} w(X_j)}\mathbbm{1}_{V^{Z_{n+1}}_i \leq v^*_i}  < \alpha$. Let $\alpha_1$, $\alpha_2$ be the values of  $ \sum^{n+1}_{i=1}\frac{w(X_i)}{\sum^{n+1}_{j=1} w(X_j)}\mathbbm{1}_{V^{Z_{n+1}}_i \leq v^*_i}$  attained at $\tilde{\alpha}_1$, $\tilde\alpha_2$, and let $\tilde\alpha = \left\{\begin{array}{cc}\tilde \alpha_1&w.p.\;\frac{\alpha - \alpha_2}{\alpha_1-\alpha_2}\\
\tilde\alpha_2 & w.p.\;\frac{\alpha_1 - \alpha}{\alpha_1-\alpha_2}
\end{array}\right.$.  Then, $\bP\left\{V^{Z_{n+1}}_{n+1} \leq Q(\tilde\alpha;\hat \mF^{Z_{n+1}} ) \right\}= \alpha$.
\label{thm:tightness}
\end{theorem}

When the score function is fixed, we can come up with a decision rule that does not depend on $y$.
\begin{theorem}
\label{thm:localQuantileFix}
Suppose Assumption \ref{ass:ass1} - \ref{ass:ass2} hold.  Let $V(.)$ to be a fixed function. For any $\tilde\alpha$, define $\bar v^* = Q(\tilde\alpha; \hat \mF)$ and  $v^*_{i1}  = Q(\tilde\alpha; \sum^n_{j=1}p^{H}_{i,j} \delta_{V_j}+p^{H}_{i,n+1}\delta_{\bar v^*}),\;v^*_{i2} = Q(\tilde\alpha  ; \sum^n_{j=1}p^{H}_{i,j} \delta_{V_j}+p^{H}_{i,n+1}\delta_0)$. If  $\bar v^* = \infty$ or if
\begin{align*}
\label{eq:goal2w}
&\sum^{n}_{i=1}\frac{w(X_i)}{\sum^{n+1}_{j=1} w(X_j)}\mathbbm{1}_{V_i \leq v^*_{i1}}  \geq \alpha,\\
&\sum^{n}_{i=1}\frac{w(X_i)}{\sum^{n+1}_{j=1} w(X_j)}\mathbbm{1}_{V_i \leq v^*_{i2}} +\frac{w(X_{n+1})}{\sum^{n+1}_{j=1} w(X_j)}\geq \alpha .  \tag{$G2^w$}
\end{align*}
Then we have $P(V_{n+1} \leq  Q(\tilde\alpha;\hat \mF) \geq \alpha$.
\end{theorem}
Corollary \ref{cor:invertFix} is a direct application of Theorem \ref{thm:localQuantileFix}.
\begin{corollary}
\label{cor:invertFix}
In the setting of Theorem \ref{thm:localQuantileFix},  let $z_{n+1} = (X_{n+1}, y)$ and $\hC(X_{n+1}) \coloneqq \{y:V(z_{n+1})\leq Q(\tilde\alpha; \hat \mF)\}$. If $\bar v^* = \infty$ or if  eq.(\ref{eq:goal2w}) holds, then we have $P(Y_{n+1}\in \hC(X_{n+1})) \geq \alpha$.
\end{corollary}
More concretely, to accommodate to the covariate shift, we need only to consider a weighted evaluation equations in Theorem \ref{thm:split2}/Algorithm \ref{alg:alg1} (fixed score function) and Theorem \ref{thm:general1} (data-dependent score function) :
\begin{enumerate}
\item In Theorem \ref{thm:split2}//Algorithm \ref{alg:alg1}, we change eq.\ (\ref{eq:goal2}) into
\begin{align*}
&\sum^{n}_{i=1}\frac{w(X_i)}{\sum^{n+1}_{j=1} w(X_j)}\mathbbm{1}_{V_i \leq v^*_{i1}}  \geq \alpha,\\
&\sum^{n}_{i=1}\frac{w(X_i)}{\sum^{n+1}_{j=1} w(X_j)}\mathbbm{1}_{V_i \leq v^*_{i2}} +\frac{w(X_{n+1})}{\sum^{n+1}_{j=1} w(X_j)}\geq \alpha . 
\end{align*}
\item Theorem \ref{thm:general1} , we change  eq.\ (\ref{eq:goal3}) into 
\[
\sum^{n+1}_{i=1}\frac{w(X_i)}{\sum^{n+1}_{j=1} w(X_j)}\mathbbm{1}_{V^{Z_{n+1}}_i \leq v^*_i}  \geq \alpha.
\]
\end{enumerate}

Under the covariate shift, localized conformal inference may help to limit the influence of samples with extremely large  weight $w(X_i)$. If $\tilde P_X$ and $P_X$ are not close to each other, the (weighted) conformal prediction may construct a CI strongly influenced  by a few samples with extremely large $w(X_i)$, even though $X_{n+1}$ can be far from those $X_i$.

To illustrate this, let $Y_i = X_i+\epsilon_i$, with $\epsilon_i \sim N(0,1)$ for $i = 1,\ldots, n+1$, and $X_{i}\sim N(0,1)$ for $i = 1,\ldots, n$, $X_{n+1}\sim N(3, 1)$. Consider the score function $V(x, y) = |y -x|$ and let the training sample size be $n = 500$. We compare conformal prediction and localized conformal prediction, both under covariate shift. For localized conformal prediction, we use a nearest-neighbor based localizer:
\[
H(x_1, x_2, X) = w(x_2)\mathbbm{1}_{\{|w(x_2) - w(x_1)|\leq Q(\frac{h}{n+1}; \sum^{n+1}_{i=1} \delta_{|w(X_i)-w(x_1)|})\}}
\]
We let $h = 450$ to limit the influence of the training samples with extreme weights on $X_{n+1}$ far away from them. We repeat the experiment 10 times and plot the constructed confidence bands using both methods for $x \leq 2$ in  Figure \ref{fig:illustrationE2}. We overlap the localized conformal bands and  the conformal bands, and observe that localized inference leads to less volatile CIs for  test samples  in this regime. For $x>2$, localized conformal prediction can produce wider CIs and more CIs with infinity lengths compared with the conformal prediction. This might actually be desirable since there are very few training samples with $x>2$, and in practice, we could want a wide/infinity CI to characterize the lack in training samples at corresponding regions. 
\begin{figure}
\caption{\em Conformal inference (blue) and localized conformal inference with automatically chosen $h$ (red) at level $\alpha = .95$.  The localized inference leads to less volatile CIs for samples that are close to the training. }
\label{fig:illustrationE2}
\begin{center}
\includegraphics[width = .5\textwidth, height = .45\textwidth]{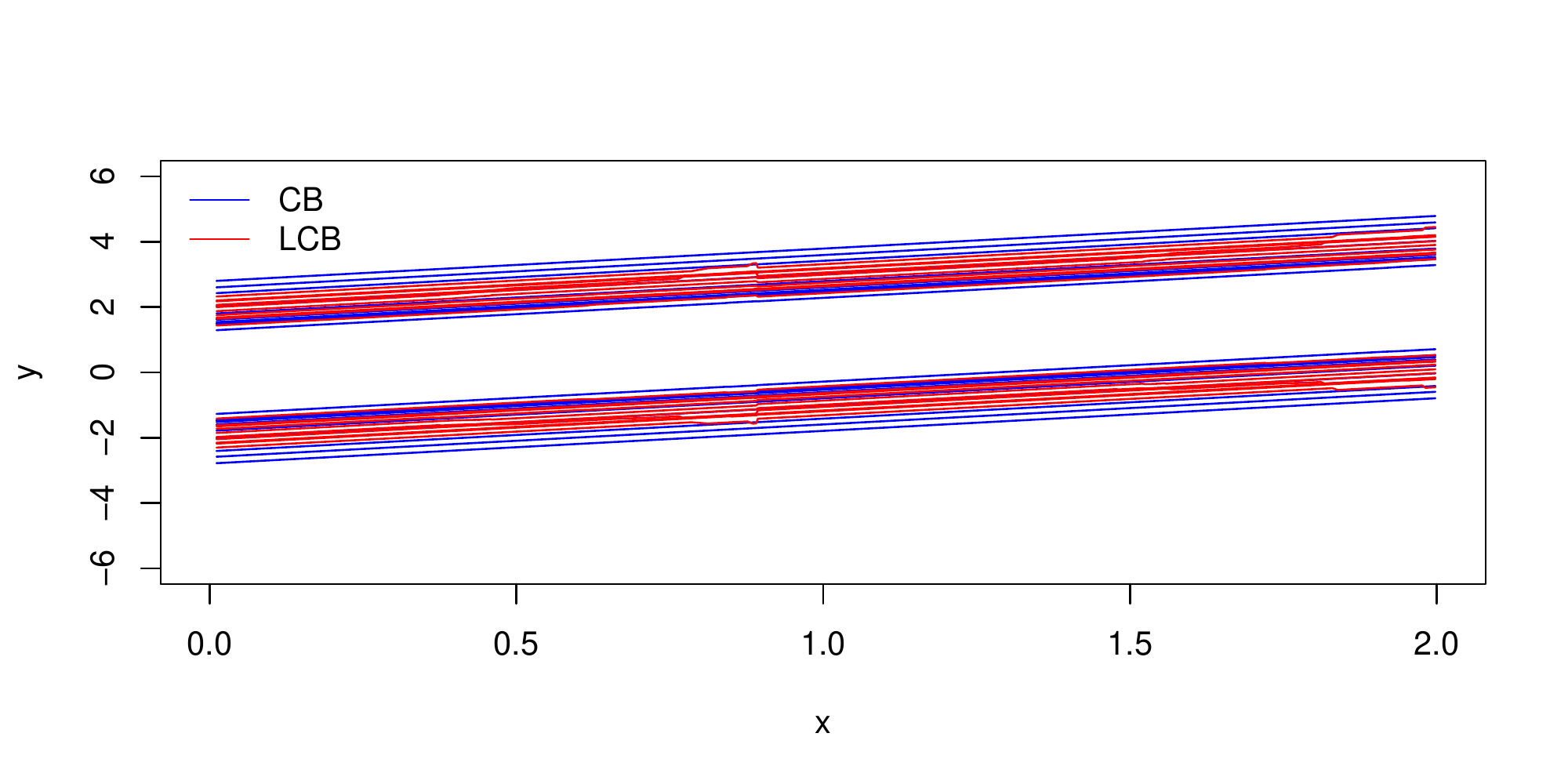} 
\end{center}
\end{figure}

\section{Proofs}
\label{app:proofs}
Theorems \ref{thm:split2}, \ref{thm:general1} and \ref{thm:general2}  are special cases of Theorems \ref{thm:localQuantileFix} ,\ref{thm:localQuantile} and \ref{thm:tightness} when $w(X) = 1$.   Hence, we will prove the later in this section.  This section is organized as following. We first give proofs to Proposition \ref{prop:equivalence}. We then give proofs to \ref{thm:localQuantile} and \ref{thm:tightness} with the help from Lemma \ref{lem:infequiv} and \ref{lem:conditional}. We prove next Theorems \ref{thm:localQuantileFix} by showing that it can guarantee the worst case scenario in Theorem \ref{thm:localQuantile} for any pre-fixed score function. We present proofs to Theorem \ref{thm:approximate} by the end of this section.
\subsection{Proof of Proposition \ref{prop:equivalence}}
\begin{proof}[Proof of Proposition \ref{prop:equivalence}]
When  $H_{i,j} = 1$ and $\tilde\alpha = \alpha$, we know $\bar v^* = Q(\alpha; V_{1:n}\cup \{\infty\})$,$v^*_{i1} = Q(\alpha; V_{1:n}\cup \{\bar v^*\})$ and $v^*_{i2} = Q(\alpha ; V_{1:n}\cup \{0\})$, $\forall i = 1,\ldots, n$.  Without loss of generality, suppose $V_1\leq V_2\leq \ldots\leq V_n$ and $\bar v^* = V_{\lceil (n+1)\alpha\rceil}$.   We show that we must have $\bar v^* = \infty$ or eq.\ (\ref{eq:goal2}). If $\bar v^* < \infty$,  then, we have $\lceil(n+1)\alpha\rceil\leq n$, and
\begin{enumerate}
\item  If  $v^*_{i,1} = v < \bar v^*$, then, $\bar v^*$ and $\{V_{\lceil (n+1)\alpha\rceil},  V_{\lceil (n+1)\alpha\rceil+1}, \ldots, V_n\}$ are both greater than $v$. Thus, $v$ is at most $\frac{ \lceil (n+1)\alpha\rceil-1}{n+1} < \alpha$ quantile of the empirical distribution $V_{1:n}\cup \{\bar v^*\}$, which is a contradiction. On the other hand, by definition of $\bar v^*$, we know
\[
\sum^{n}_{i=1}\frac{1}{n+1}\mathbbm{1}_{V_i \leq \bar v^*} +\frac{1}{n+1}\mathbbm{1}_{ \infty\leq \bar v^*} =   \sum^{n}_{i=1}\frac{1}{n+1}\mathbbm{1}_{V_i \leq \bar v^*} \geq \alpha.
\]
Hence, $\sum^{n}_{i=1}\frac{1}{n+1}\mathbbm{1}_{V_i \leq v^*_{i1}}  \geq \alpha$.
\item It is easy to check that $v^*_{i,2} = Q([\alpha-\frac{1}{n+1}]_+\frac{n}{n+1}; V_{1:n})$.  Hence,  $v^*_{i,2}$ is the $\lceil [\alpha-\frac{1}{n+1}]_+ \frac{n^2}{n+1}\rceil$ smallest  value in $\{V_1,\ldots, V_n\}$. Consequently, we have
\begin{align*}
\sum^{n}_{i=1}\frac{1}{n+1}\mathbbm{1}_{V_i \leq v^*_{i2}} +\frac{1}{n+1}&= \frac{\lceil [\alpha-\frac{1}{n+1}]_+ \frac{n^2}{n+1}\rceil+1}{n+1} \\
&\geq \frac{\lceil \tilde\alpha \frac{n^2}{n+1}+\frac{(n+1)^2-n^2}{n+1}\rceil}{n+1}\geq \frac{\lceil \alpha (n+1)\rceil}{n+1}.
\end{align*}

\end{enumerate}
Combine them together, we know that $\tilde\alpha = \alpha$ leads to $\bar v^* = \infty$ or eq.\ (\ref{eq:goal2}), and Theorem \ref{thm:split2} recovers the result that 
\[
\bP\left\{V_{n+1} \leq Q(\alpha; V_{1:n}\cup \{\infty\} ) \right\}\geq \alpha.
\]
\end{proof}

\subsection{Proofs Theorems of  \ref{thm:localQuantile} and \ref{thm:tightness}}
Lemma \ref{lem:infequiv} and Lemma \ref{lem:conditional} are important components for the proofs of Theorems of  \ref{thm:localQuantile} and \ref{thm:tightness}.
\begin{lemma}
\label{lem:infequiv}
For any $\alpha$ and sequence $\{V_1,\ldots, V_{n+1}\}$, we have
\[
V_{n+1} \leq Q(\alpha; \sum^n_{i=1} p_i \delta_{V_i} + p_{n+1}\delta_{V_{n+1}}) \Leftrightarrow V_{n+1} \leq Q(\alpha; \sum^n_{i=1} p_i \delta_{V_i} + p_{n+1}\delta_\infty),
\]
where $\sum^n_{i=1} p_i \delta_{V_i} + p_{n+1}\delta_{V_{n+1}}$ and $\sum^n_{i=1} p_i \delta_{V_i} + p_{n+1}\delta_\infty$ are some weighted empirical distributions with weights $p_i\geq 0$ and $\sum^{n+1}_{i=1} p_i = 1$.
\end{lemma}
\begin{proof}[Proof of Lemma \ref{lem:infequiv}]
By definition, we know 
\[
V_{n+1} \leq Q(\alpha; \sum^n_{i=1} p_i \delta_{V_i} + p_{n+1}\delta_{V_{n+1}}) \Rightarrow V_{n+1} \leq Q(\alpha; \sum^n_{i=1} p_i \delta_{V_i} + p_{n+1}\delta_\infty).
\]
To show that Lemma \ref{lem:infequiv} holds, we only need to show that 
\[
V_{n+1} >Q(\alpha; \sum^n_{i=1} p_i \delta_{V_i} + p_{n+1}\delta_{V_{n+1}})\Rightarrow V_{n+1} > Q(\alpha; \sum^n_{i=1} p_i \delta_{V_i} + p_{n+1}\delta_\infty).
\]
Without loss of generality, we assume $0 = V_{0} \leq V_1 \leq V_2 \leq \ldots \leq V_n$, and consider the case where $V_{n+1} >Q(\alpha; \sum^n_{i=1} p_i \delta_{V_i} + p_{n+1}\delta_{V_{n+1}})$.

In this case, we must have $\sum^n_{i=1} p_i \geq \alpha$, and the empirical lower $\alpha$ quantile is the smallest index $i$ such that $\sum^{i}_{j=1} p_j \geq \alpha$. Let $i^* \leq n$ be this index. Since $V_{n+1} > V_{i^*}$ and  $\sum^{i}_{j=1} p_j \geq \alpha$, by definition, we know
\begin{align*}
&\sum^n_{i=1}\mathbbm{1}_{V_i \leq V_{i^*}} \geq \alpha \Leftrightarrow Q(\alpha; \sum^n_{i=1} p_i \delta_{V_i} + p_{n+1}\delta_\infty)\leq V_i^*\\
\Rightarrow&V_{n+1} > Q(\alpha; \sum^n_{i=1} p_i \delta_{V_i} + p_{n+1}\delta_\infty).
\end{align*}
\end{proof}
\begin{lemma}
\label{lem:conditional}
For any event 
\[
\mathcal{T} \coloneqq \left\{\{Z_i, i = 1,\ldots, n+1\}= \{z_i\coloneqq (x_i, y_i), i = 1,\ldots, n+1\}\right\},
\]
we have
\[
\bP\{V^{Z_{n+1}}_{n+1} \leq Q(\tilde\alpha; \sum^{n+1}_{i=1} p^{H}_{n+1,i} \delta_{V^{Z_{n+1}}_{i}}) |\mathcal{T}\}= \bE\left\{\sum^{n+1}_{i=1}\frac{w(x_i)}{\sum^{n+1}_{j=1} w(x_j)}\mathbbm{1}_{v_i \leq v^*_i}|\mathcal{T}\right\},
\]
where $v_i = V(z_i, (z_1,\ldots, z_n, z_{n+1}))$,  $v^*_{i} =  Q(\tilde\alpha; \sum^{n+1}_{j=1}p^{H}_{i,j} \delta_{V^{Z_{n+1}}_j})$ for $i = 1,2,\ldots, n+1$, and $\tilde\alpha = \tilde\alpha(Z)$ can be dependent of the data of through the set $Z$ where $Z =\{Z_1,\ldots, Z_{n+1}\}$. The expectation on the right-hand-side is taken over the randomness of $\tilde\alpha$ conditional on $\mathcal{T}$.
\end{lemma}
\begin{proof}[Proof of Lemma \ref{lem:conditional}]
Let $\sigma$ be a permutation of numbers $1,2,\ldots, n+1$.   We know that
\begin{align*}
P(\sigma_{n+1} = i|\mathcal{T}) &= \frac{w(x_i)\#\{\sigma:\sigma_{n+1} = i\}}{\sum^{m+1}_{j=1}w(x_j)\#\{\sigma:\sigma_{n+1} = j\}} = \frac{w(x_i)}{\sum^{n+1}_{j=1}w(x_j)}.
\end{align*}
Also, since the function $V(., Z) = V(.)$ and the localizer $H(., ., X) = H(.,.)$ have fixed function forms conditional on $\mathcal{T}$, and $\tilde\alpha$ (can be random) is independent of the data conditional $\mathcal{T}$, we also have
\begin{align*}
&\bP(V^{Z_{n+1}}_{n+1} \leq Q(\tilde\alpha; \sum^{n+1}_{i=1} p^{H}_{n+1,i} \delta_{V^{Z_{n+1}}_{i}}) |\mathcal{T}, \tilde\alpha) \\
=& \sum^{n+1}_{i=1} P(\sigma_{n+1} = i|\mathcal{T})\mathbbm{1}_{\{V_{n+1}\leq v^*_{n+1}(\sigma)|\mathcal{T}, \sigma_{n+1} = i\}}\\
=& \sum^{n+1}_{i=1} \frac{w(x_i)}{\sum^{n+1}_{j=1}w(x_j)}\mathbbm{1}_{\{v_i\leq v^*_{n+1}(\sigma)|\mathcal{T}, \sigma_{n+1} = i\}}
\end{align*}
where $v^*_{i}(\sigma) = Q(\tilde\alpha; \sum^n_{j=1}p^{H}_{\sigma_i,\sigma_j}\delta_{v_{\sigma_j}})$ is the realization of $v^*_i$ with data permutation $\sigma$ conditional on $\mathcal{T}$ and $\tilde\alpha$: 
\[
v^*_i(\sigma) = Q(\tilde\alpha; \sum^{n+1}_{k=1} \frac{H( x_{\sigma_i}, x_{\sigma_k})}{\sum^{n+1}_{j=1}H( x_{\sigma_i}, x_{\sigma_j})}\delta_{v_{\sigma_k}} )
\]
With a slight abuse of notation, we let $v^*_i$ corresponds to the case where $\sigma_i = i$.  We immediately observe that
\begin{equation}
\label{eq:app1}
v^*_i(\sigma) = v^*_{\sigma_i}
\end{equation}
Consequently, we have $\bP\{V^{Z_{n+1}}_{n+1} \leq v^*_{n+1} |\mathcal{T}, \tilde\alpha\} =  \sum^{n+1}_{i=1} \frac{w(x_i)}{\sum^{n+1}_{j=1}w(x_j)}\mathbbm{1}_{\{v_i\leq v^*_i\}}$. Marginalize over $\tilde\alpha|\mathcal{T}$, we have
\[
\bP\{V^{Z_{n+1}}_{n+1} \leq v^*_{n+1} |\mathcal{T}, \tilde\alpha\}  = \bE\{ \sum^{n+1}_{i=1} \frac{w(x_i)}{\sum^{n+1}_{j=1}w(x_j)}\mathbbm{1}_{\{v_i\leq v^*_i\}}|\mathcal{T}\}
\]
\end{proof}

\subsubsection{Proof of Theorem \ref{thm:localQuantile}}
Define 
\[
\mathcal{T} \coloneqq \left\{\{Z_i, i = 1,\ldots, n+1\}= \{z_i\coloneqq (x_i, y_i), i = 1,\ldots, n+1\}\right\}.
\]
When we choose $\tilde\alpha$ such that eq.(\ref{eq:goal1w}) is satisfied, this decision rule does not depend on the ordering of data conditional on $\mathcal{T}$: for any permutation $\sigma$ of numbers $1, 2,\ldots, n+1$, we have
\begin{align*}
\sum^{n+1}_{i=1}\frac{w(X_i)}{\sum^{n+1}_{j=1} w(X_j)}\mathbbm{1}_{V^{Z_{n+1}}_i \leq v^*_i}|\mathcal{T},\sigma &=  \sum^{n+1}_{i=1}\frac{w(x_{\sigma_i})}{\sum^{n+1}_{j=1} w(x_{\sigma_j})}\mathbbm{1}_{v_{\sigma_i} \leq v^*_{\sigma_i}} \\
&= \sum^{n+1}_{i=1}\frac{w(x_{i})}{\sum^{n+1}_{j=1} w(x_{j})}\mathbbm{1}_{v_{i} \leq v^*_{i}}.
\end{align*}
Since $V(., Z) $ and $H(., ., X) $ are fixed functions conditional on $\mathcal{T}$ (see the arguments for eq.(\ref{eq:app1}) in Lemma \ref{lem:conditional}). Hence, apply Lemma \ref{lem:conditional}, we have
\[
\bP\{V^{Z_{n+1}}_{n+1} \leq Q(\tilde\alpha; \hat \mF^{Z_{n+1}}_{n+1}) |\mathcal{T}\}= \bE\left\{\sum^{n+1}_{i=1}\frac{w(x_i)}{\sum^{n+1}_{j=1} w(x_j)}\mathbbm{1}_{v_i \leq v^*_i}|\mathcal{T}\right\} \geq \alpha.
\]
Marginalize over $\mathcal{T}$, we have
\[
\bP\{V^{Z_{n+1}}_{n+1}\leq  Q(\tilde\alpha; \hat \mF^{Z_{n+1}}_{n+1}) \} \geq \alpha.
\]
By Lemma \ref{lem:infequiv}, equivalently, we also have
\[
\bP\{V^{Z_{n+1}}_{n+1}\leq Q(\tilde\alpha; \hat \mF^{Z_{n+1}})\} \geq \alpha.
\]

\subsubsection{Proof of Theorem \ref{thm:tightness}}
Define
\[
\mathcal{T} \coloneqq \left\{\{Z_i, i = 1,\ldots, n+1\}= \{z_i\coloneqq (x_i, y_i), i = 1,\ldots, n+1\}\right\}.
\]
Following the same argument as used for $\tilde\alpha$ in the proof of Theorem \ref{thm:localQuantile}, we know that both $\tilde\alpha_1$, $\tilde\alpha_2$ and $\alpha_1$, $\alpha_2$ are fixed conditional on $\mathcal{T}$.  As a result,  when $\tilde\alpha = \left\{\begin{array}{cc}\tilde \alpha_1&w.p.\;\frac{\alpha - \alpha_2}{\alpha_1-\alpha_2}\\
\tilde\alpha_2 & w.p.\;\frac{\alpha_1 - \alpha}{\alpha_1-\alpha_2}
\end{array}\right.$, we know that $\tilde\alpha$ is independent of the data conditional on $\mathcal{T}$. Apply Lemma \ref{lem:conditional}, we have
\begin{align*}
\bP\{V^{Z_{n+1}}_{n+1} \leq Q(\tilde\alpha; \hat \mF^{Z_{n+1}}_{n+1}) |\mathcal{T}\}&= \bE\left\{\sum^{n+1}_{i=1}\frac{w(x_i)}{\sum^{n+1}_{j=1} w(x_j)}\mathbbm{1}_{v_i \leq v^*_i}|\mathcal{T}\right\} \\
& = \alpha_1\frac{\alpha-\alpha_2}{\alpha_1 -\alpha_2}+\alpha_2\frac{\alpha_1 - \alpha}{\alpha_1 - \alpha_2} = \alpha.
\end{align*}
Marginalize over $\mathcal{T}$, we have
\[
\bP\{V^{Z_{n+1}}_{n+1} \leq Q(\tilde\alpha; \hat \mF^{Z_{n+1}}_{n+1})\}= \alpha.
\]
By Lemma \ref{lem:infequiv}, equivalently, we have
\[
\bP\{V^{Z_{n+1}}_{n+1} \leq Q(\tilde\alpha; \hat \mF^{Z_{n+1}})\}= \alpha.
\]
\subsection{Proof of Theorem \ref{thm:localQuantileFix}}
Before we proceed to the proof of Theorem \ref{thm:localQuantileFix}, we first introduce Lemma \ref{lem:bound}.
\begin{lemma}
\label{lem:bound}
Let $Z_{n+1} = (X_{n+1}, y)$, let  $v^*_{i}  = Q(\tilde\alpha; \sum^n_{j=1}p^{H}_{i,j} \delta_{V_j}+p^{H}_{i,n+1}\delta_{V_{n+1}}),\;\forall i = 1,2,\ldots, n+1$,  and $\hat \alpha(y) \coloneqq  \sum^{n+1}_{i=1}\frac{w(X_i)}{\sum^{n+1}_{j=1j}w(X_j)}\mathbbm{1}_{V_i \leq v^*_i}$ for $v^*_i$ evaluated at $Z_{n+1} = (X_{n+1},y)$.  For any $\tilde\alpha$ such that eq.\ (\ref{eq:goal2w}) holds, we have $\min_y \hat \alpha(y)\geq \alpha$.
\end{lemma}
\begin{proof}[Proof of Lemma \ref{lem:bound}]
The key observations which  we use to prove Lemma \ref{lem:bound} are that, for any  $\tilde\alpha$,  $y$ only influences $v^*_{i}$ through $V_{n+1}$. 
\begin{itemize}
\item $v^*_i$ is non-decreasing as $V_{n+1}$ increases. Thus,   $\sum^{n}_{i=1}\frac{w(X_i)}{\sum^{n+1}_{j=1j}w(X_j)}\mathbbm{1}_{V_i \leq v^*_i}$ is non-decreasing as $V_{n+1}$ increases.
\item  $\bar v^{*}=v^*_{n+1}$ if $ V_{n+1} > \bar v^{*}$:  If $\bar v^* = \infty$, we have $\bar v^{*}=v^*_{n+1}$. Otherwise, the quantile $Q(\tilde\alpha; \hat \mF)$ takes value in $\{V_1, \ldots, V_n\}$, and suppose it is the $(i^*)^{th} (\leq n)$ smallest value in $\{V_1,\ldots, V_n\}$.  Without loss of generality, suppose $V_1\leq V_2\leq \ldots\leq V_n$.  By definition, $i^* $ is the smallest number such that 
\[
\sum^{i^*}_{i=1} \frac{p^{H}_{n+1,i} \delta_{V_j}}{\sum^{n+1}_{j=1}p^{H}_{n+1,j} \delta_{V_j}}\geq \tilde\alpha.
\]
On the one hand, according to the definition of $Q(\tilde\alpha; \hat \mF)$, we have $Q(\tilde\alpha; \hat \mF) \leq V_{i^*}$. Hence, $v^*_{n+1}\geq \bar v^*$. On the other hand, we always have   $Q(\alpha; \hat \mF)\geq Q(\alpha, \sum^{n+1}_{j=1}p^{H}_{n+1,j} \delta_{V_j})$. Consequently,  we have $\bar v^{*}=v^*_{n+1}$.
\end{itemize}
This leads us to consider the following two cases:
\begin{enumerate}
\item If  $\bar v^{*} < V_{n+1}$, use the fact that $v^*_i$ is non-decreasing in $V_{n+1}$ and $v^*_{n+1} = \bar v^*$, we have
\begin{align*}
\inf_{\bar v^{*} <V_{n+1}\leq \infty} \hat \alpha(y)&= \inf_{V_{n+1} > \bar v^{*}} \sum^{n}_{i=1}\frac{w(X_i)}{\sum^{n+1}_{j=1j}w(X_j)}\mathbbm{1}_{V_i \leq v^*_i}\\
&\geq \sum^{n}_{i=1}\frac{w(X_i)}{\sum^{n+1}_{j=1}w(X_j)}\mathbbm{1}_{V_i \leq v^*_{i1}}.
\end{align*}
\item If $V_{n+1}\leq \bar v^* $, again by the non-decreasing nature of $ \sum^{n}_{i=1}\frac{w(X_i)}{\sum^{n+1}_{j=1j}w(X_j)}\mathbbm{1}_{V_i \leq v^*_i}$, we have
\[
\inf_{V_{n+1} \leq v^*_{n+1} <\infty} \hat \alpha(y)  = \sum^{n}_{i=1}\frac{w(X_i)}{\sum^{n+1}_{j=1j}w(X_j)}\mathbbm{1}_{V_i \leq v^*_{i2}} + \frac{w(X_{n+1})}{\sum^{n+1}_{j=1j}w(X_j)}.
\]
\end{enumerate}
Combine them together, we have
\begin{align*}
\inf_{y} \hat \alpha(y)& \geq \min( \sum^{n}_{i=1}\frac{w(X_i)}{\sum^{n+1}_{j=1j}w(X_j)}\mathbbm{1}_{V_i \leq v^*_{i1}}, \sum^{n}_{i=1}\frac{w(X_i)}{\sum^{n+1}_{j=1j}w(X_j)}\mathbbm{1}_{V_i \leq v^*_{i2}}+\frac{w(X_{n+1})}{\sum^{n+1}_{j=1} w(X_j)}) .
\end{align*}
\end{proof}
We now prove Theorem \ref{thm:localQuantileFix} using  Lemma \ref{lem:conditional}  and Lemma \ref{lem:bound}.
\begin{proof}[Proof of Theorem \ref{thm:localQuantileFix}]
Let $\mathcal{T} = \{(Z_i, i = 1,\ldots, n+1) = (z_i, i = 1,\ldots, n+1)\}$ be the set of values for $Z_1, \ldots, Z_{n+1}$, where $z_i = (x_i,y_i)$ for $i= 1,\ldots, n+1$.  Let $\sigma_{1:(n+1)}$ be a permutation of $\{1,\ldots, n+1\}$.  By Lemma \ref{lem:bound}, although $\tilde\alpha$ does not depend on $y_{n+1}$, we can still achieve
\[
 \sum^{n+1}_{i=1}\frac{w(X_{i})}{\sum^{n+1}_{j=1j}w(X_{j})}\mathbbm{1}_{V_{i} \leq v^*_{i, n+1}} \geq \alpha
\]
where $v^*_{i,n+1} = Q(\tilde\alpha; \hat \mF_i)$ and $\tilde\alpha$ is a value we found (based on some pre-fixed procedure) satisfying eq.\ (\ref{eq:goal2w}). Note that $\tilde\alpha$ is not symmetric on the observations $Z_1, \ldots, Z_{n+1}$, and it assigns $Z_{n+1}$ a special role. Hence, we can not directly apply Lemma \ref{lem:conditional}. To use Lemma \ref{lem:conditional}, we first apply Lemma \ref{lem:bound} to permuted observations, which leads to the eq.\ (\ref{eq:bound2}):
\begin{equation}
\label{eq:bound2}
 \sum^{n+1}_{i=1}\frac{w(X_{\sigma_{i}})}{\sum^{n+1}_{j=1j}w(X_{\sigma_j})}\mathbbm{1}_{V_{\sigma_i} \leq v^*_{\sigma_i, \sigma_{n+1}}} \geq \alpha
\end{equation}
where $v^*_{\sigma_i, \sigma_{n+1}} = Q(\tilde\alpha^{\sigma_{n+1}}; \hat \mF_{\sigma_{i}})$, and $\tilde\alpha^{\sigma_{n+1}}$ is a value for $\tilde\alpha$ such that eq.\ (\ref{eq:goal2w}) holds with the permutation order $\sigma$. Since eq.\ (\ref{eq:bound2}) holds for any permutation $\sigma$, and the permutation only influence it via $\tilde\alpha^{\sigma_{n+1}}$, consider all the possibilities for $\sigma_{n+1}$, e.g., $\sigma_{n+1} = 1,\ldots, n+1$, we have
\begin{equation}
\label{eq:bound3}
 \sum^{n+1}_{i=1}\frac{w(X_{\sigma_{i}})}{\sum^{n+1}_{j=1j}w(X_{\sigma_j})}\mathbbm{1}_{V_{\sigma_i} \leq u^*_{\sigma_i}} \geq \alpha
\end{equation}
where $u^*_{\sigma_i} =  Q(\min^{n+1}_{l=1}\tilde\alpha^{l}; \hat \mF_{\sigma_{i}})$. The new quantity $\min^{n+1}_{l=1}\tilde\alpha^{l}$ depends only on $\{Z_1,\ldots, Z_{n+1}\}$ but not their ordering, thus, we can  combine eq.\  (\ref{eq:bound3}) with Lemma \ref{lem:conditional} to prove Theorem \ref{thm:localQuantileFix}:
\begin{align*}
\bP\{V_{n+1} \leq Q(\tilde\alpha;\hat \mF_{n+1})|\mathcal{T}\}&\geq \bP\{V_{n+1} \leq Q(\min^{n+1}_{i=1} \tilde\alpha^{n+1};\hat \mF_{n+1})|\mathcal{T}\}\\
&= \bE\{ \sum^{n+1}_{i=1} \frac{w(x_i)}{\sum^{n+1}_{j=1}w(x_j)}\mathbbm{1}_{\{v_i\leq u^*_i\}}|\mathcal{T}\} \geq \alpha
\end{align*}
The above holds for arbitrary value set $\mathcal{T}$, hence, marginalizing over  all possible values of $z_i$ for $i=1,\ldots, n$ and $x_{n+1}$, we have  $\bP\{V_{n+1} \leq Q(\tilde\alpha;  \hat \mF)\} \geq \alpha$.
\end{proof}
\subsection{Proof of  Theorem \ref{thm:approximate}}
\subsubsection{Part (a)}
\begin{proof}
Conditional on $X_{n+1} = x_0$, define $\tilde p(x)=  \frac{H(x_0,x)}{\sum^{n}_{j=1}H(x_0,X_i)+H(x_0,\tilde X_{n+1})}$ and let 
\[
\widetilde C(\tilde X_{n+1},x_0)\coloneqq \{y: V(\tilde X_{n+1},y)\leq Q(\alpha; \sum^{n}_{i=1}\tilde p(X_i)\delta_{V_i}+\tilde p(\tilde X_{n+1}) \delta_{\infty})\}.
\]
As a direct application of Proposition \ref{prop:prop2}, we have 
\[
\bP\{\tilde Y_{n+1} \in \widetilde C(\tilde X_{n+1},x_0)\}\geq \alpha.
\]

Since the $H(x_0, x_0)\geq H( x_0, \tilde X_{n+1})$, define $p(x)=  \frac{H(x_0,x)}{\sum^{n}_{j=1}H(x_0,X_i)+H(x_0,x_0)}$, we have
\[
Q(\alpha; \sum^{n}_{i=1}\tilde p(X_i)\delta_{V_i}+\tilde p(\tilde X_{n+1}) \delta_{\infty}) \leq Q(\alpha; \sum^{n}_{i=1}p(X_i)\delta_{V_i}+p(x_0) \delta_{\infty}) .
\]
Hence, let $\hat C(x_0)\coloneqq \{y: V(\tilde X_{n+1},y)\leq Q(\alpha; \sum^{n}_{i=1}p(X_i)\delta_{V_i}+p(x_0) \delta_{\infty})\}$, we have
\[
\bP\{\tilde Y_{n+1} \in \hat C(x_0)\} \geq \alpha.
\]
The above is true for all $x_0\in \real^p$, thus, $\bP\{\tilde Y_{n+1} \in \hC( X_{n+1})|X_{n+1} = x_0\}\geq \alpha$ for all $x_0$.
\end{proof}

\subsubsection{Part (b)}
\begin{proof}
From part (a), we know that  for any $x_0$, we have
\[
\bP\{\tilde Y_{n+1} \in \hC( X_{n+1})|X_{n+1} = x_0\}\geq \alpha
\]
Conditional on $X_{n+1}=x_0$, let $M = \int K(\frac{x-x_0}{h})d P_X(x)$ be the normalization constant for the distribution of $\tilde X_i$ and $\frac{d P^{x_0}(x)}{d x} = \frac{1}{M} K(\frac{x-x_0}{h}) \frac{d P(x)}{d x}$.  Let $\mu(.)$ be the joint distribution of $Z_{1:n}$ after reweighting. Then, we have:
\begin{align*}
\alpha\leq&\bP\{\tilde Y_{n+1} \in \hC( X_{n+1})|X_{n+1} = x_0\}\\
=&\int_{z_{1:n}}\left(\int_{\tilde x_{n+1}}\left(\int_{\tilde y_{n+1}} \mathbbm{1}_{\tilde y_{n+1}\in \hC(x_0)}p_{Y|X}(\tilde y_{n+1}|\tilde x_{n+1}) d\tilde y_{n+1}\right) d P^{x_0}(x)\right) d \mu(z_{1:n})\\
=&\frac{1}{M}\int_{z_{1:n}}\left(\int_{\tilde x_{n+1}}\left(\int_{\tilde y_{n+1}} \mathbbm{1}_{V(x, y)\leq Q(\alpha, \hat \mF)}p_{Y|X}(\tilde y_{n+1}|\tilde x_{n+1}) d\tilde y_{n+1}\right)K(\frac{x-x_0}{h})d P(x)\right)d \mu(x_{1:n})
\end{align*}
By the Lipschitz assumption, we know
\begin{align*}
\alpha\leq&\frac{1}{M} \int_{X_{1:n}}\left(\int_{\tilde X_{n+1}}\left(\int_{\tilde Y_{n+1}} \mathbbm{1}_{V(x, y)\leq Q(\alpha, \hat \mF)}p_{Y|X}(y|x_0)dy\right)K(\frac{x-x_0}{h})d P_X(x)\right)d \mu(x_{1:n})\\
 +&\frac{L}{M}\int \|x-x_0\| K(\frac{x-x_0}{h})d P_X(x)\\
=& P(Y_{n+1}\in \hat C(X_{n+1})|X_{n+1}=x_0)+\frac{L}{M}\int \|x-x_0\| K(\frac{x-x_0}{h})d P_X(x)
\end{align*}
For a Gaussian kernel $K(\frac{x-x_0}{h}) = \frac{1}{(2\pi h^2)^{\frac{d}{2}}}\exp(-\frac{\|x-x_0\|_2^2}{h^2})$ and under the regularity condition for $P_X(.)$, we know that $b_1\leq M\leq b_2$ and 
\begin{align*}
&\int \|x-x_0\| K(\frac{x-x_0}{h})d P_X(x)  \leq \frac{b_2 }{(2h^2)^{\frac{d}{2}}}\frac{1}{\Gamma(\frac{d}{2}+1)}\int^{\infty}_{r=0} r^d\exp(-\frac{r^2}{2h^2})  d r =  \frac{hb_2}{(d+1)\Gamma(\frac{d}{2}+1)2^{\frac{d}{2}}}
\end{align*}

Hence, if $h\rightarrow 0$, we have $[\alpha - P(Y_{n+1}\in \hC(X_{n+1})|X_{n+1}=x_0)]_+\rightarrow 0$.
\end{proof}

\section{Choice of $H$}
\label{app:choiceH}
We consider two types of localizers in this paper:
 \begin{enumerate}
 \item  Distance based localizer
 \[
 H_h(x_1, x_2,X)  = \mathbbm{1}_{\{|\frac{x_2- x_1}{h}| \leq 1\}}.
 \]
 \item Nearest-neighbor based localizer
 \[
 H_h(x_1, x_2,X) = \mathbbm{1}_{\{|x_1-x_2| \leq Q(\frac{h}{n}; \sum^{n+1}_{i=1}\delta_{|X_i - x_1|}) \}}.
 \]
 \end{enumerate}
 In practice, we can pick $h$ beforehand based on a date set $\mathcal{D}_0$ that is independent of $Z=\{Z_1,\ldots, Z_n, Z_{n+1}\}$, with $Z^0_i\overset{i.i.d}{\sim}\mathcal{P}$ for $Z^0_i=(X^0_i, Y^0_i)\in \mathcal{D}_0$, $i = 1,\ldots, m$. Let $X^0 = \{X^0_1,\ldots, X^0_m\}$. 
 
Define the score for sample $Z^0_i$ as $V^0_i = V(Z^0_i)$ if $V(.)$ is also independent of $\mathcal{D}_0$. If  $V(.)$ is trained using $\mathcal{D}_0$, we suggest to let $V^0_i$ be its score from cross-validation using $\mathcal{D}_0$.  For example , suppose $V(z) = |y - \hat \mu(x)|$, where $\hat \mu(.)$ is the prediction function trained using $\mathcal{D}_0$, we can let
\[
V^0_i = |Y^0_i - \hat \mu^{-i}(X^0_i)|
\]
where $\hat \mu^{-i}(X^0_i)$ is the trained prediction function with a subset in $\mathcal{D}_0\setminus\{Z^0_i\}$. 

Based on the discussion in section \ref{sec:method_h},  we want to  to tradeoff between locality and volatility, and choose $h$ to have relatively narrow and stable CIs for most of the samples.   Let $\mathcal{X}$ be a subset of $\mathcal{D}_0$. We suggest to pick $h$ such that in $\mathcal{X}$: (1) the average length for CI is small, (2) the average variance of lengths of CIs  conditional on $x$ is small, and (3) the coverage is at least $\alpha$ for the constructed CI in $\mathcal{X}$.  We consider the subset  $\mathcal{X}$ instead of  every sample in $\mathcal{D}_0$ because, for the distance based localizer, it is okay if we have a  small portion of samples with $\infty$-length CI. In this case, we can compare choices of $h$ based  on those points with finite length CIs by considering the samples in the subset  subset  $\mathcal{X}$ .  We do require  the subset $\mathcal{X}$ to be large though, for example, by default, we let $\mathcal{X}$ contain $90\%$ of the samples, and if $h$ leads to more than 10\% of CIs being $\infty$, it is always not preferred.

More specifically, let $h_1 < h_2 < \ldots < h_L$, we use the following steps to choose $h$ from $h_l$ for $1\leq l\leq L$ automatically using $\mathcal{D}_0$.  To reduce the computational complexity, we simply let $\tilde\alpha = \alpha$ in Algorithm \ref{alg:alg1}.
\begin{enumerate}
\item Let $\bar v^*_{i,l} $ be the realization of $\bar v^*$ at $\tilde\alpha = \alpha$, with test sample $Z^0_i$ and training samples $\mathcal{D}_0\setminus\{Z^0_i\}$, and with parameter $h_l$ for the localizer $H$: $\bar v^*_{i,l} = Q(\alpha; \sum_{j\neq i} p^{l}_{i, j}\delta_{V^0_{j}}+p^{l}_{i, i}\delta_{\infty})$, here $p^{l}_{i,j} = \frac{H_{h_l}(X^0_i, X^0_j, X^0)}{\sum^m_{j=1}H_{h_l}(X^0_i, X^0_j, X^0)}$.
\item As $h$ becomes smaller, the percent of $\bar v^*_{i,l}$ being $\infty$ may becomes higher for $i = 1,\ldots, m$ (note that if $\bar v^*_{i,l_1} = \infty$, then, for $l_2 < l_1$, $\bar v^*_{i,l_2} = \infty$ ). We consider only those $h_l$ that result in less than $(1-\omega)$ percent of $\infty$, and let $\mathcal{X}\subseteq \mathcal{D}_0$ be the intersection of samples with finite $\bar v^*_{i,l}$ for all $h_l$ we consider.
\item Let $s_l=\frac{\sum^m_{i=1}\bar v^*_{i,l}\mathbbm{1}_{X^0_i\in \mathcal{X}}}{\sum^m_{i=1}\mathbbm{1}_{X^0_i\in \mathcal{X}}}$ be an estimate of average CI length in $\mathcal{X}$ using $h_l$.
\item Let  $\gamma_l =  \frac{(1-\alpha)\sum^m_{i=1}\mathbbm{1}_{X^0_i\in \mathcal{X}}}{\sum^m_{i=1}\mathbbm{1}_{\{X^0_i\in \mathcal{X}, V^0_i > \bar v^*_{i,l}\}}}\vee 1$ be a measure of degree of empirical under-coverage. (If the empirical coverage for samples in $\mathcal{X}$ is at least $\alpha$, $\gamma_l = 1$; otherwise, $\gamma_l > 1$.)
\item We estimate the average standard deviation with Bootstrap: for each sample $X^0_i$ and $h = h_l$, let $\bar v^{b,*}_{i,l}$, $b = 1,\ldots, B$, be the value $\bar v^*$ with test sample $X^0_i$ and $(n-1)$ training samples $Z^0_j$ bootstrapped from $\mathcal{D}_0$ with their corresponding score values $V^{0}_j$. Let $\sigma_{i,l}$ be the estimated standard deviation using those $\bar v^{b,*}_{i,l}$  with finite values for $b = 1,\ldots, B$,  and let $\sigma_l$ be the  average standard deviation of $\sigma_{i,l}$ across $i = 1,\ldots, m$.
\item Choose $h$ as $h^* = \arg\min_{h\in \{h_1,\ldots, h_L\}}\left( \gamma_l\times (s_l+\sigma_l)\right)$.
\end{enumerate}
By default, we let $\omega = .9$ and $B = 20$.  In high-dimension where $p$ is large,  instead of applying the localizer to the raw feature $x$,  we usually will prefer to use a low dimensional function $t:\real^p\rightarrow \real^K$, and apply $H$ to $t(x)$. How to  find a good $t$ is non-trivial and beyond the scope of this paper. Here, to illustrate that the localized conformal prediction still gain over the conformal prediction if we can approximately find the low dimensional direction where  the score function has high variability, we consider a simulated high dimensional example and simply let $t(x) = x_j$ where $j$ is the direction that leads to the largest mutual information  between $V^0_i$ and  $X^0_{i,j}$, $i = 1,\ldots, m$.

\begin{example}
\label{exm2}
Let $Y_i = X^T_i\beta+\epsilon_i$, with $\beta = (\underbrace{1,\ldots, 1}_{3},\underbrace{0,\ldots, 0}_{p-3})^T$,  $X_{i,j}\sim Unif[-3,3]$  for $i = 1,\ldots, n+1$ and $j = 1,\ldots, p$, and we consider two cases of error distribution: (a) $\epsilon_i \overset{i.i.d}{\sim} N(0,1)$, and (b)$\epsilon_i |X_i\sim \left\{\begin{array}{cc}.5N(0,1)& |X_{i,p}|\leq 1\\ 2N(0,1) &|X_{i,p}| > 1\end{array}\right.$. 
We let $V(x, y) = |y - \mu(x)|$, where $\mu(x)$ is the prediction model $\mu(x)$ trained using  cross-validation lasso regression on a data set $\mathcal{D}_0$ of size $n = 500$. We use an independent set $\mathcal{D}_1$ of size $n = 500$ to perform the conformal inference and localized conformal inference. For  localized conformal inference, we use both the distance based localizer $H_h^1$ and the nearest-neighbor based localizer $H_h^2$ with the tuning parameter $h$ automatically chosen using $\mathcal{D}_0$. We perform 1000 experiments for $p = 3$ and $p = 500$. We see that all three constructions have controlled the coverage in Table  \ref{tab:exm2}. In Figure \ref{fig:exm2}, we plot the constructed CIs at $\alpha = .95$ for $V_i$ using different methods and the true values of $V_i$ across 1000 repetitions.  We plot the constructed lower and upper boundaries of CIs against their feature values $X_{i,p}$, and we can see that localized conformal predictions with both the distance-based and the nearest neighbor based localizers (red and blue dots) have captured the underlying heterogeneity of CIs for across different $X_{i,p}$.
\begin{table}[H]
\centering
\caption{Example \ref{exm2}. Coverage at $\alpha = .95$.}
\label{tab:exm2}
\begin{tabular}{|l|rr|rr|}
  \hline
$\alpha = .95$ & $p = 3$ &  &$p = 500$ & \\ 
   \hline
  & (a) & (b) & (a) & (b) \\ 
  \hline
CB & 0.95 & 0.95 & 0.94 & 0.96 \\ 
  LCB1& 0.95 & 0.95 & 0.95 & 0.96 \\ 
  LCB2 & 0.95 & 0.96 & 0.94 & 0.96 \\ 
   \hline
\end{tabular}
\end{table}
 \end{example}
 
\begin{figure}
\caption{\em Example \ref{exm2}. Confidence bands constructed using 1000 repetitions with targeted level at $\alpha = .95$. The black,  blue, red and green dots respectively  represent (1) actual $V_i$ for the test samples (error), (2) the conformal inference (CB) for $V_i$, (3)the localized conformal inference for $V_i$ with distance based localizer $H_h^1$ (LCB1), and (4) the localized conformal inference with nearest-neighbor based localizer $H_h^2$ (LCB2). The x-axis shows values for $X_p$, and y-axis shows values for the upper and lower boundaries of constructed CIs for each of the test samples.}
\label{fig:exm2}
\begin{center}
\includegraphics[width = .8\textwidth, height = 1\textwidth]{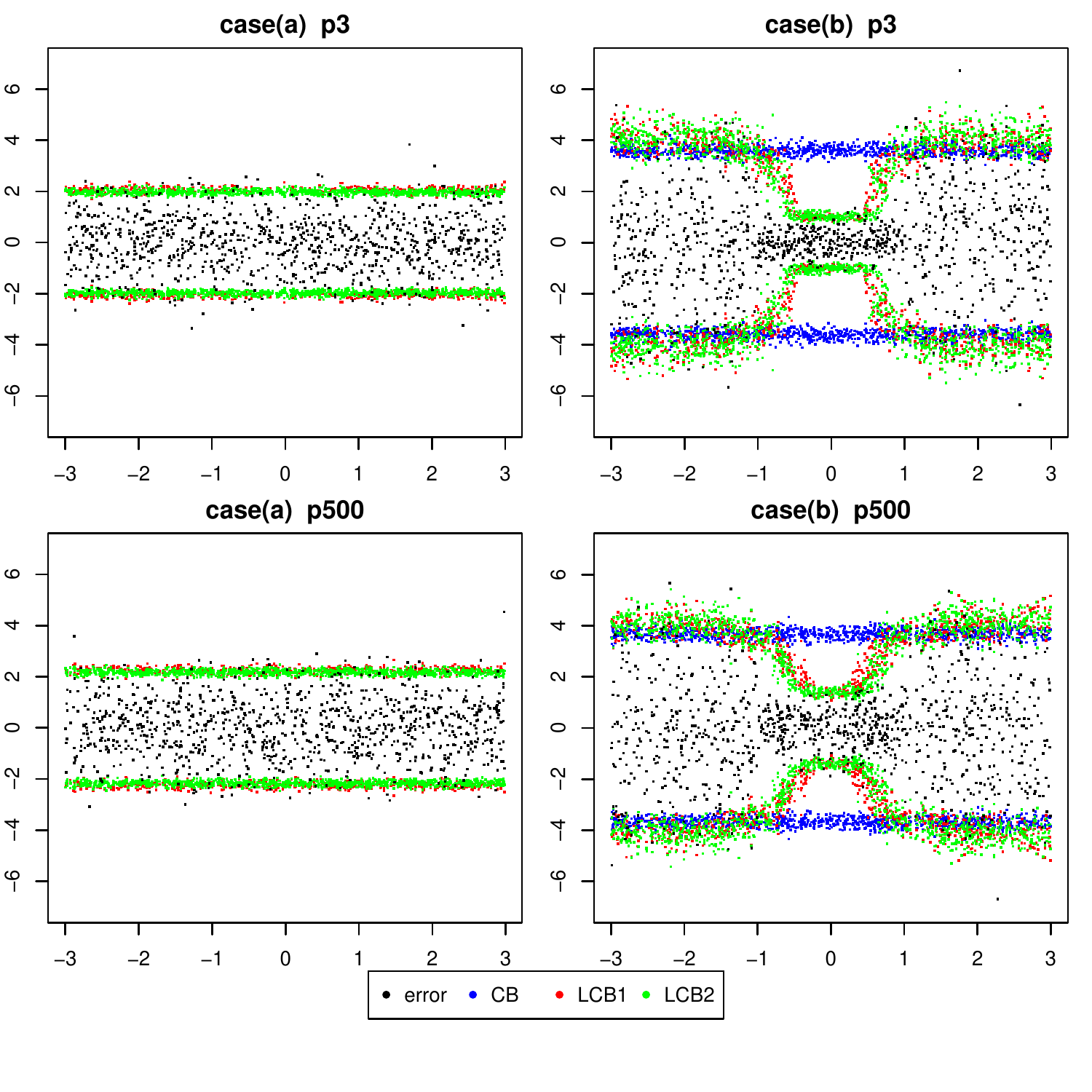} 
\end{center}
\end{figure}

\bibliographystyle{unsrtnat}
\bibliography{distributionShift}
\end{document}